\pdfoutput=1 

\documentclass[12pt]{extarticle}
\usepackage[utf8]{inputenc}
\usepackage[T1]{fontenc}

\usepackage{caption}
\oddsidemargin \evensidemargin
\usepackage{float}
\usepackage{resizegather}
\usepackage{amssymb}
\usepackage{pmboxdraw}
\usepackage{amsmath}
\usepackage{nccmath}
\usepackage{bbm}
\usepackage{amsthm}
\usepackage{dsfont}
\usepackage[dvipsnames]{xcolor}
\usepackage{graphicx}
\usepackage{alphabeta}
\usepackage{tikz,pgfplots}
\pgfplotsset{compat=1.18}
\usepackage{pgfplots}
\usepackage{comment}
\usepackage[margin=0.5cm]{caption}
\usepackage{hyperref}
\usepackage{subcaption}
\usepackage[all]{xy}
\usepackage[greek,english]{babel}
\usepackage[export]{adjustbox}
\usepackage{hyperref}
\usepackage{mathrsfs}  
\hypersetup{
    colorlinks=true,
    linkcolor=myblue,
    filecolor=magenta,      
    urlcolor=myblue,
    citecolor=BurntOrange,
}
\usepackage{enumitem}
\usepackage{booktabs}
\usepackage[all]{xy}

\definecolor{myblue}{rgb}{0.25,0.45,0.99}
\definecolor{myorange}{rgb}{0.8500, 0.3250, 0.0980}
\definecolor{myyellow}{rgb}{0.9290, 0.6940, 0.1250}
\definecolor{mypurple}{rgb}{0.4940, 0.1840, 0.5560}
\definecolor{mygreen}{rgb}{0.4660, 0.6740, 0.1880}

\usepackage[activate={true,nocompatibility},final,tracking=true,kerning=true,spacing=true,factor=1100,stretch=10,shrink=10]{microtype}

\usepackage{listings}
\lstdefinelanguage{Macaulay2}{
  morekeywords={matrix, apply, entries, sum, product, ideal, flatten, print, factor, subsets, WeylAlgebra, holonomicRank, DsingularLocus, gens, QQ},
  sensitive=true,
  morecomment=[l]{--},
  morestring=[b]",
}

\usepackage{algorithm}
\usepackage{algpseudocode}

\newtheorem{theorem}{Theorem}[section]
\newtheorem{theorem*}[theorem]{Theorem*}
\newtheorem{question}[theorem]{Question}
\newtheorem{lemma}[theorem]{Lemma}
\newtheorem{corollary}[theorem]{Corollary}
\newtheorem{proposition}[theorem]{Proposition}

\newtheorem{thm/conj}[theorem]{Theorem/Conjecture}

\newtheorem{conjecture}[theorem]{Conjecture}

\theoremstyle{definition}
\newtheorem{examplex}[theorem]{Example}
\newenvironment{example}
{\pushQED{\qed}\examplex}
{\popQED\endexamplex}

\newtheorem{definition}[theorem]{Definition}
\newtheorem{remark}[theorem]{Remark}

\theoremstyle{remark}
\newcommand{\Erec}{E_L}
\renewcommand{\P}{\mathbb{P}}
\newcommand{\C}{\mathbb{C}}

\DeclareMathOperator{\FL}{F}
\def\Omegarel#1{\Omega_{\rm rel}^{#1}}

\newcommand\restr[2]{{
  \left.\kern-\nulldelimiterspace 
  #1
  \vphantom{\big|} 
  \right|_{#2}
  }}

\usepackage{titling} 
\date{\today}
\usepackage{cleveref}

\let\OLDthebibliography\thebibliography
\renewcommand\thebibliography[1]{
  \OLDthebibliography{#1}
  \setlength{\parskip}{0.5pt}
  \setlength{\itemsep}{0.8pt plus 0.5ex}
}

\title{\textbf{Principal Matroid Determinants}}
\author{Saiei-Jaeyeong Matsubara-Heo and Simon Telen}

\begin{document}

\maketitle

\begin{abstract}
\noindent We develop a theory of principal determinants and hypergeometric systems for realizable matroids. Our framework parallels the toric theory of Gel'fand, Kapranov, and Zelevinsky (GKZ), but with the combinatorics of matroids and their flats replacing the usual role of polytopes and their faces. 
In this analogy, the toric variety is replaced by a reciprocal linear space. The {principal $A$-determinant} is replaced by the {principal matroid determinant}, defined as a specialization of a resultant. 
The GKZ hypergeometric system is replaced by the {matroid hypergeometric system}, a holonomic $D$-module of combinatorial nature whose singular locus is conjectured to be the principal matroid~determinant.
\end{abstract}

\section{Introduction}

The seminal book \cite{GKZbook} by Gel'fand, Kapranov and Zelevinsky develops a theory of discriminants, resultants and determinants in the language of modern algebraic geometry. Central concepts in that book include projective duality (discriminants) and Chow forms (resultants). There is considerable emphasis on the case of toric varieties, where the \emph{principal $A$-determinant} $E_A$ plays a key role \cite[Chapter 10]{GKZbook}. Among the crucial results in \cite{GKZbook} are the factorization of $E_A$ in terms of $A$-discriminants, and the fact that the Newton polytope of $E_A$ is the secondary polytope. The latter relates principal $A$-determinants to the theory of regular triangulations, which lies at the heart of tropical geometry \cite{maclagansturmfels}. Another salient feature of $E_A$ is that its vanishing locus is the Euler discriminant of a family of hypersurfaces with a fixed monomial support \cite{esterov2013discriminant}. This means that $E_A$ vanishes precisely when the Euler characteristic of such a hypersurface is different from its generic value. Recently, Euler discriminants have found applications in algebraic statistics \cite{amendola2019maximum} and particle physics \cite{DlapaHelmerPapathanasiouTellander,fevola2024principal}.

In this paper, we propose a matroid analog of the GKZ theory.
Consider a $d$-dimensional linear space $L \subset \mathbb{P}^n$. 
The hyperplane arrangement $L \cap \{x_0x_1 \cdots x_n = 0 \}$ in $L$ gives rise to a matroid $M(L)$. 
Our main object is the \emph{principal matroid determinant} $E_L$, a polynomial whose factorization and Newton polytope are determined by the combinatorics of $M(L)$. 
It shares many of its features with the principal $A$-determinant, see Theorems \ref{thm:mainintro} and \ref{thm:mainEA}. 

To define the principal matroid determinant, we introduce some notation. 
Let $T = (\mathbb{C}^\times)^{n+1} / \mathbb{C}^\times \simeq (\mathbb{C}^\times)^n$ be the dense torus of the $n$-dimensional complex projective space $\mathbb{P}^n$. That is, $T \subset \mathbb{P}^n$ consists of all points $(x_0:  \cdots: x_n) \in \mathbb{P}^n$ which can be represented by $n+1$ nonzero homogeneous coordinates $x_i$. For any $k \in \mathbb{Z}$, we define the map 
\begin{equation} \label{eq:phi_k} \phi_k : \mathbb{P}^n \dashrightarrow \mathbb{P}^n, \quad \phi(x_0:\cdots:x_n) \, = \, (x_0^k : \cdots : x_n^k). \end{equation}
For a projective variety $X \subseteq \mathbb{P}^n$, we shall write $X^k = \phi_k(X)$ for $k \geq 0$, and $X^k = \overline{\phi_k(X \cap T)}$ for $k < 0$. 
The reciprocal linear space $L^{-1} = \overline{\phi_{-1}(L \cap T)}$ consists of the coordinate-wise reciprocals of points in $L$.
Such reciprocal linear spaces are central in our story. Their coordinate ring and connection to matroids were studied in \cite{proudfoot2006broken}. The definition of $E_L$ also involves $L^{-2}$.

We now apply the classical {$X$-discriminant} construction from \cite{GKZbook} to $X = L^{-1}$. 
Let $\nabla(L^{-1}) \subset (\mathbb{P}^n)^\vee$ be the projective dual variety of $L^{-1}$. We call $\nabla(L^{-1})$ the matroid discriminant variety of $L$. When $\nabla(L^{-1})$ is a hypersurface, we write $\Delta(L^{-1})\in \mathbb{C}[z_0, \ldots, z_n]$ for its defining equation. If $\nabla(L^{-1})$ is not a hypersurface, then we set $\Delta(L^{-1}) = 1$. 

\begin{example}[$d = 2, \, n = 3$] \label{ex:intro}
    Consider the reciprocal row span $L^{-1}$ of the $3\times 4$~matrix
    \begin{equation}\label{eq:generic matroid}
         A \, = \,  \begin{pmatrix}
        -1 &1 & 0 & 0 \\ -1 & 0 & 1 & 0  \\ 
        -1 & 0 & 0 & 1 
    \end{pmatrix}; \quad L^{-1} \, = \, \overline{\{(x_0^{-1}:x_1^{-1}:x_2^{-1}:x_3^{-1}) \in  \mathbb{P}^3\, : \, x \in {\rm Row}(A) \cap T \} }.
        \end{equation}
    This is {Cayley's cubic surface}, defined by the equation $x_0x_1x_2 + x_1x_2x_3 + x_0x_2x_3+x_0x_1x_3 = 0$. 
    Its dual surface $\nabla(L^{-1})$ is a quartic surface in $(\mathbb{P}^3)^\vee$, known as the {Steiner surface}. 
    Using coordinates $z_0, z_1, z_2, z_3$ on $(\mathbb{P}^3)^\vee$, the defining equation of $\nabla(L^{-1})$ is $\Delta(L^{-1}) = 0$, where  
    \begin{equation}\label{eq:matroid discriminant for a uniform matroid}
     \hspace{-0.9cm} 
     \Delta(L^{-1}) \, = \,
     \sum_{i=0}^3z_i^4-4\sum_{i\neq j}z_i^3z_j+6\sum_{i<j}z_i^2z_j^2+4\sum_{i\neq j\neq k\neq i}z_i^2z_jz_k-40\,z_0z_1z_2z_3.
     \end{equation}
    The varieties $L^{-1}$ and $\nabla(L^{-1})$ are shown in Figure \ref{fig:cayleysteiner}, in the charts $\sum_i x_i = 1$ and $\sum_i z_i = 1$. 
\end{example}

For each flat $F \subset \{0, \ldots, n\}$ of the matroid $M(L)$ represented by $L$, we write $L_F$ for the projection of $L$ to the coordinate subspace indexed by $F$. The matroid discriminant variety of $L_F$ is $\nabla(L_F^{-1}) \subset (\mathbb{P}^{|F|-1})^\vee$. Its defining polynomial $\Delta(L_{F}^{-1}) \in \mathbb{C}[z_i \, : \, i \in F]$ is expressed in the $z$-variables which are indexed by $F$. If $F = \{i\}$ is a singleton, then we set $\Delta(L_{F}^{-1}) = z_i$. 
Finally, let ${\cal R}(L^{-2}) \subset {\rm Mat}_{d+1,n+1}(\mathbb{C})$ be the $L^{-2}$-resultant hypersurface, as defined in \cite[Chapter 3, Section 2B]{GKZbook}. Its defining equation is a polynomial in the entries $u_{ij}, \, i = 0, \ldots, d, \, j = 0, \ldots, n$ of a generic $(d+1) \times (n+1)$ matrix $U$, denoted by ${\rm Res}(L^{-2}) \in \mathbb{C}[u_{ij}]$. 

The principal matroid determinant $\Erec \in \mathbb{C}[z_0, \ldots, z_n]$ is obtained by specializing ${\rm Res}(L^{-2})$ as follows: $U = A_z := A \cdot {\rm diag}(z_0, \ldots, z_n)$. Here $A \in \mathbb{C}^{(d+1) \times (n+1)}$ is a matrix whose projectivized row span is $L$, and ${\rm diag}(z_0, \ldots, z_n)$ is a diagonal matrix with diagonal entries $z_i$. 
\begin{theorem} \label{thm:mainintro}
    Let $L \subset \mathbb{P}^n$ be a linear space of dimension $d< n$, not contained in a coordinate hyperplane. The principal matroid determinant $\Erec$ has the following properties. 
    \begin{enumerate}
    \itemsep0em 
        \item $\Erec$ has degree $(d+1)\, 2^{d-c(L)+1} \, \mu(L)$, where $c(L)$ is the number of connected components of the matroid~$M(L)$ and $\mu(L) = \deg(L^{-1})$ is its M\"obius invariant. (Theorem~\ref{thm:degErec}) 
        \item The radical of $\Erec$ is the product of matroid discriminants $\Delta(L_{F}^{-1})$, where $F$ ranges over all flats of the matroid $M(L)$. (Theorem \ref{thm:factorization})
        \item The Newton polytope of $E_L$ is a generalized permutohedron. (Theorem \ref{thm:newtEL})
        \item The zero locus of $\Erec$ consists of all points $z \in (\mathbb{P}^n)^\vee$ such that the topological Euler characteristic of the hyperplane section $\{ x \in L^{-1} \cap T \, : \, z_0 x_0 + \cdots + z_n x_n = 0 \}$ differs from its generic value. (Theorem \ref{thm:ED})
    \end{enumerate}
\end{theorem}

\begin{example} \label{ex:introcontd}
    For $L^{-1}$ as in \eqref{eq:generic matroid}, the surface $L^{-2} \subset \mathbb{P}^3$ has degree 12. Its defining polynomial~is 
    \[ f_{4,4,4,0} -4 \, (f_{4,4,3,1} + f_{4,4,1,3} +f_{4,3,4,1} ) + 6 \, (f_{4,4,2,2} + f_{4,2,4,2}) + 4 \, (f_{4,3,3,2} +f_{4,3,2,3} + f_{4,2,3,3}) - 40 \, f_{3,3,3,3},\]
    where $f_{i,j,k,\ell}$ is the sum of all monomials in the orbit of $x_0^ix_1^jx_2^kx_3^\ell$ under the action of the cyclic group, e.g., $f_{4,4,4,0} = x_0^4 x_1^4 x_2^4 + x_0^4 x_1^4 x_3^4 + x_0^4 x_2^4 x_3^4 + x_1^4 x_2^4 x_3^4$. The resultant polynomial ${\rm Res}(L^{-2})$ is a polynomial in the entries of a generic $3 \times 4$ matrix 
    \[ U \, = \, \begin{pmatrix}
        u_{00} & u_{01} & u_{02} & u_{03} \\ 
        u_{10} & u_{11} & u_{12} & u_{13} \\ 
        u_{20} & u_{21} & u_{22} & u_{23}
    \end{pmatrix}.\]
    It vanishes on all such matrices for which there exists $x \in L^{-2}$ satisfying $U \cdot x = 0$, where $x = (x_0, x_1, x_2, x_3)$ is thought of as a column vector of length $4$. One obtains ${\rm Res}(L^{-2})$ by substituting the signed $3 \times 3$ minors of $U$ into the defining equation of $L^{-2}$. The result is an irreducible polynomial of degree 36 in $u_{ij}$. It has $2 \, 129 \, 137$ terms. The principal matroid determinant is obtained by specializing this polynomial as follows: substitute $U$ by
    \[ \begin{pmatrix}
        -1 &1 & 0 & 0 \\ -1 & 0 & 1 & 0  \\ 
        -1 & 0 & 0 & 1 
    \end{pmatrix} \cdot {\rm diag}(z_0, z_1, z_2, z_3) \, = \, \begin{pmatrix}
        -z_0 & z_1 & 0 & 0 \\ 
        -z_0 & 0 & z_2 & 0  \\ 
        -z_0 & 0 & 0 & z_3 
    \end{pmatrix}.\]
    The result is the reducible polynomial $\Erec = z_0^8 \, z_1^8 \, z_2^8 \, z_3^8 \, \Delta(L^{-1})$. 

    This computation matches Theorem \ref{thm:mainintro}. We have $d = 2, c(L) = 1$ and $\mu(L) = \deg (L^{-1}) = 3$, which gives $\deg(\Erec) = 36$. The uniform matroid $M(L)$ has ten non-empty flats: four of rank one, six of rank two and one of rank three. The reciprocal linear spaces $L_{F}^{-1}$ corresponding to flats of rank two are isomorphic to $\mathbb{P}^1$. Their matroid discriminants are $1$. The rank-one flats contribute the factors $z_i^8$ in $\Erec$. The Newton polytope of $E_L$ is a dilated standard simplex, which is indeed a generalized permutohedron. For Theorem \ref{thm:mainintro}(4), see Example \ref{ex:1.3}. 
    \begin{figure}
        \centering
        \includegraphics[height = 6cm]{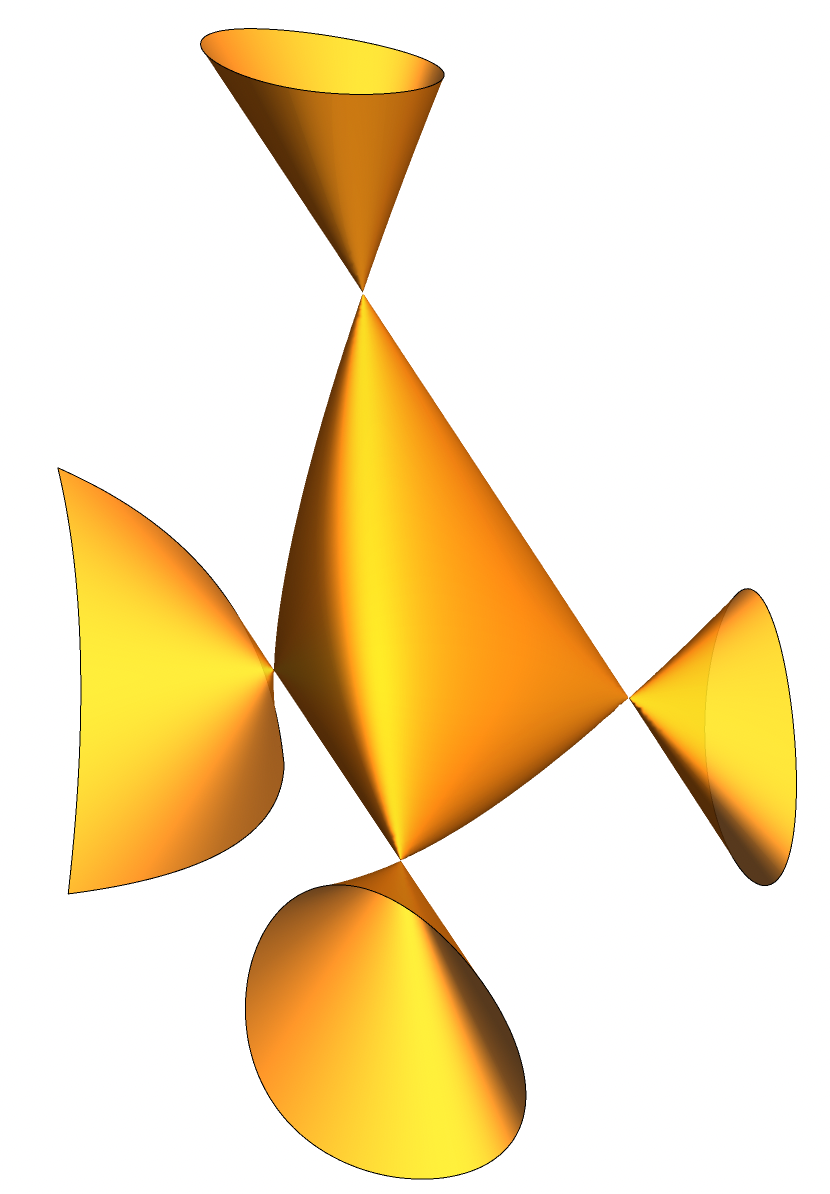} \quad \quad \quad 
        \includegraphics[height = 6cm]{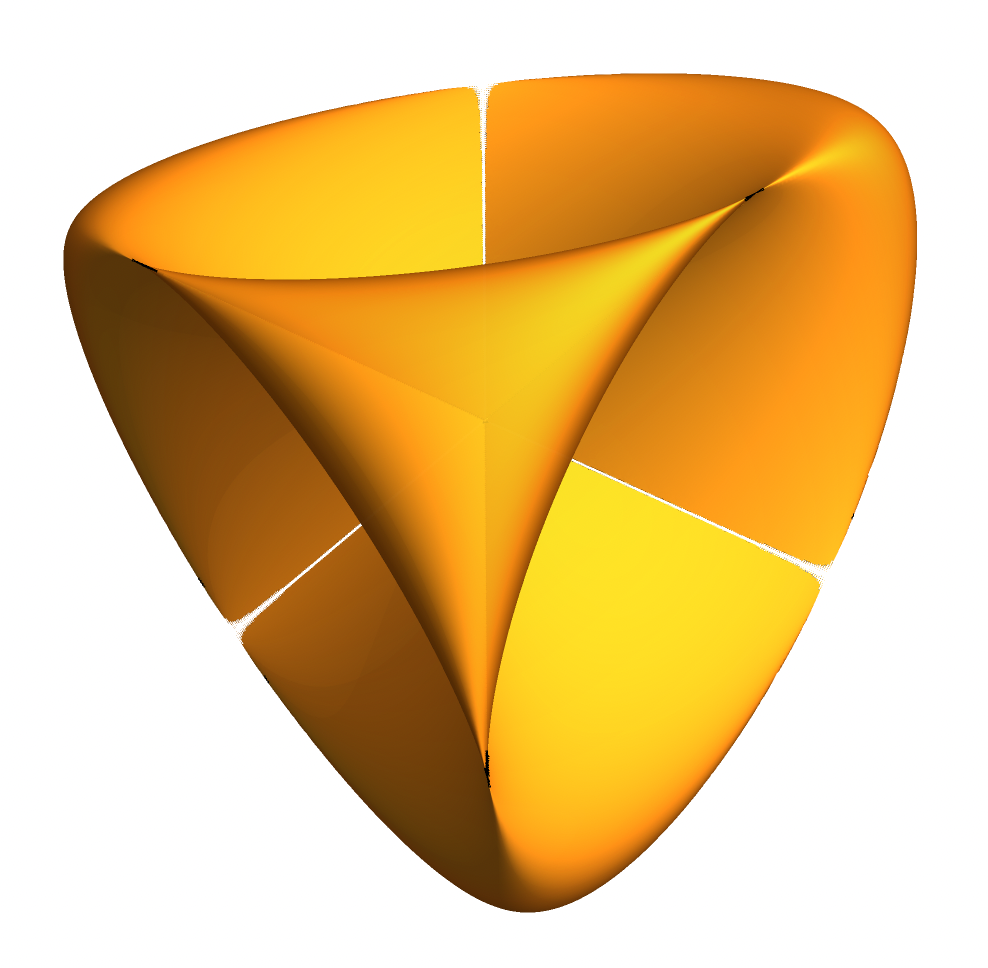} 
        \caption{The Cayley surface (left) and the Steiner surface (right) are each other's dual. }
        \label{fig:cayleysteiner}
    \end{figure}
\end{example}

We explain a different interpretation of point 4 in Theorem \ref{thm:mainintro} in terms of the hyperplane arrangement associated with our matroid $M(L)$. For $i = 0, \ldots, n$, let $\ell_i(t) = a_{0i} \, t_0 + \cdots + a_{di} \,  t_d$ be the linear forms encoded by the columns of $A$. These form an arrangement of $n+1$ hyperplanes denoted by ${\cal A} = V(\ell_0 \cdots \ell_n) \subset \mathbb{P}^d$. The map $\ell^{-1}: \mathbb{P}^d\setminus {\cal A} \rightarrow L^{-1} \cap T$ given by $t \mapsto (\ell_0(t)^{-1}: \cdots : \ell_n(t)^{-1})$ is an isomorphism. The hyperplane sections appearing in Theorem \ref{thm:mainintro}(4) are identified with a family of hypersurfaces in $\mathbb{P}^d \setminus {\cal A}$:
\begin{equation}\label{eq:family of hypersurfaces} V_z \, = \, \Big \{ t \in \mathbb{P}^d \setminus {\cal A} \, : \, \frac{z_0}{\ell_0(t)} + \cdots + \frac{z_n}{\ell_n(t)} \, = \, 0 \Big \}. \end{equation}
We have that $\ell^{-1}(V_z) = \{ x \in L^{-1} \cap T \, : \, z_0 x_0 + \cdots + z_n x_n = 0 \}$.
Theorem \ref{thm:mainintro}(4) states that $\Erec = 0$ defines the Euler discriminant locus of this family, which is a first step towards the Euler stratification from \cite{telen2024euler}. This will be proved in Section \ref{sec:6}.

\begin{figure}[t]
        \centering
        \includegraphics[width=0.9\linewidth]{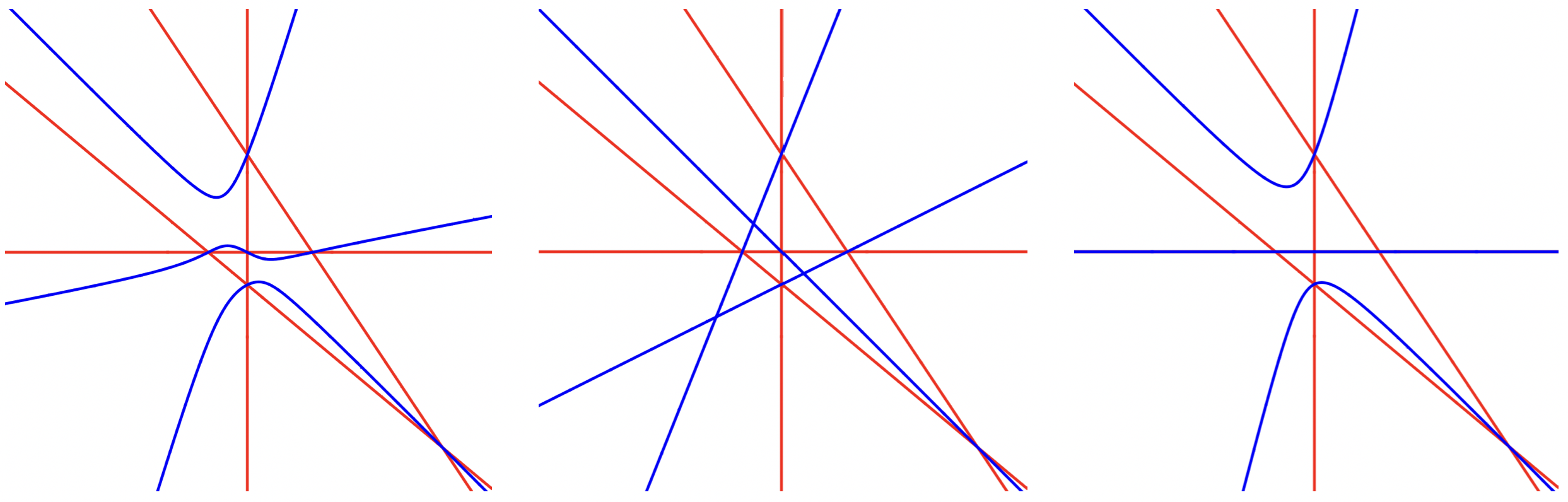}
        \caption{The cubic curve ${V}_z$ (blue) and the line arrangement of $M(L)$ (red) in the chart $3t_0+2t_1+8t_2 \neq 0$. Left: $z = (1,1,1/2,1)$, middle: $z= (1,1,1,1)$, right: $z=(1,1,0,1)$. }
        \label{fig:eulerdisc}
    \end{figure}
    
\begin{example}\label{ex:1.3}
    Using the matrix $A$ from \eqref{eq:generic matroid}, to each $(z_0,z_1,z_2,z_3) \in \mathbb{C}^4$ we associate a curve 
    \begin{equation}\label{eq:V_z}
     V_z \, = \, \Big \{t \in \mathbb{P}^2 \setminus {\cal A} \, : \, -\frac{z_0}{t_0+t_1+t_2} + \frac{z_1}{t_0} + \frac{z_2}{t_1} + \frac{z_3}{t_2}  \, = \, 0 \Big \}.
    \end{equation}
    Here ${\cal A} \subset \mathbb{P}^2$ is the arrangement of four lines given by $(t_0+t_1+t_2)t_0t_1t_2 = 0$.
    Examples for three different choices of $z$ are shown in Figure \ref{fig:eulerdisc}. The restriction of $\ell^{-1}: \mathbb{P}^2 \setminus {\cal A} \rightarrow \mathbb{P}^3$ to $V_z$ identifies our curve with $\{ x \in L^{-1} \cap T \, : \, z_0x_0 + \cdots + z_3 x_3 = 0\}$, where $L^{-1}$ is the Cayley cubic. The family of curves obtained in this way is the three-dimensional linear system of elliptic curves passing through the six intersection points of the four lines, corresponding to the six rank-two flats of $M(L)$. The Euler discriminant of this family consists of all $z \in \mathbb{P}^3$ for which the curve $V_z$ has a non-generic topological Euler characteristic. Let us give a geometric intuition of why the Euler characteristic is non-generic when any of the factors of $E_L$ vanishes. First, a generic curve in our linear system (see Figure \ref{fig:eulerdisc}, left) is a smooth elliptic curve ${\cal E}$ with six punctures. Hence, the generic value of the Euler characteristic is $\chi(V_z) = \chi({\cal E}) - 6 = -6$. For $z \in \nabla(L^{-1})$, the elliptic curve is singular. For instance, each nodal singularity increases $\chi(V_z)$ by one. In the middle part of Figure \ref{fig:eulerdisc}, we chose $z = (1,1,1,1) \in \nabla(L^{-1})$, for which $V_z$ is a union of three lines in $\mathbb{P}^2 \setminus {\cal A}$. There are three nodal singularities, and thus $\chi(V_z) = -6 + 3 = -3$. If $z_i = 0$, then our cubic equation has $\ell_i$ as a factor, so $V_z$ is a plane conic with five punctures, see Figure \ref{fig:eulerdisc} (right). The topological Euler characteristic is $\chi(V_z) = 2-5 = -3$. We note that the Euler discriminant of this example was computed in \cite[Section 5.1]{matsubara2025hypergeometric}.
\end{example}

As emphasized in the Preface of \cite{GKZbook}, the main interest of GKZ in the principal $A$-determinant $E_A$ stems from the fact that it encodes the singular locus of a holonomic $D$-module, called the {$A$-hypergeometric system} or GKZ system \cite[Corollary 1.9]{GKZbookchapter}. 
In the words of these authors, that system is a ``quantization'' of discriminants. GKZ systems have given rise to a substantial body of research in the theory of hypergeometric functions \cite{Matusevich-Miller-Walther,ReicheltSchulzeSevenheckWalther2021,Schulze_Walther}.

An important consequence of the fact that $E_A$ defines the singular locus of the GKZ system is that its vanishing locus contains all singularities of the Euler integrals studied in \cite{GKZ}.
In this article, we use recent results from \cite{matsubara2025hypergeometric} to connect the principal matroid determinant with Euler integrals of the following form: 
\begin{equation} \label{eq:eulerint_intro} g_\Gamma(z) \, = \, \int_\Gamma \left( \frac{z_0}{\ell_0(t)} + \cdots + \frac{z_n}{\ell_n(t)}\right)^{-s} \, \prod_{i = 0}^n \ell_i(t)^{u_i} \, \Omega(t).\end{equation}
Here $\Omega(t) = \sum_{i = 0}^d (-1)^i \, t_i  \bigwedge_{j \neq i} {\rm d}t_i$ and the exponents $u_i, s$ are complex numbers satisfying $u_0 + \cdots + u_n + s + d+ 1 = 0$. The integral is to be interpreted as a twisted period of the very affine variety $(\mathbb{P}^d \setminus {\cal A}) \setminus V_z$, where $V_z$ is the hypersurface defined by $\sum_i z_i \,  \ell_i^{-1}(t) = 0$ in $\mathbb{P}^d \setminus {\cal A}$, see for instance \cite[Section 3]{Matsubara-Heo2023Four}. In particular, $\Gamma$ is any twisted $d$-cycle on $(\mathbb{P}^d \setminus {\cal A}) \setminus V_z$.

We are interested in determining linear differential operators with polynomial coefficients in the variables $z$ which annihilate the function $g_\Gamma(z)$ for any choice of $\Gamma$. 
We shall write $\partial_i$ for the partial derivative with respect to $z_i$.
Let $D_{\C^{n+1}}$ be the Weyl algebra $\mathbb{C}\langle z_i,\partial_i \, ; \,  i=0,\dots,n\rangle$.
We introduce a left ideal $H_L(u)\subset D_{\C^{n+1}}$ and a left $D_{\C^{n+1}}$-module $M_{L}(u)=D_{\C^{n+1}}/H_{L}(u)$ (Definition \ref{def:matroidhypergeometricsystem}).
The $D_{\C^{n+1}}$-module $M_{L}(u)$ is the {matroid hypergeometric system}. It is our analog of the GKZ $A$-hypergeometric system in toric geometry (see Theorems \ref{thm:mainintro2} and \ref{thm:GKZsystem}).

\begin{theorem} \label{thm:mainintro2}
The matroid hypergeometric system $M_L(u)$ has the following properties.
\begin{enumerate}
    \item $M_L(u)$ is a holonomic $D_{\C^{n+1}}$-module.
    \item The singular locus of $M_L(u)$ is contained in the zero locus of $\Erec$. 
\end{enumerate}
\end{theorem}

We expect that the inclusion in Theorem \ref{thm:mainintro2}(2) is in fact an equality (Conjecture \ref{conj:bigconjecture}(3)).

\begin{example}\label{ex:Lauricella}
    For $L$ from Example \ref{ex:intro}, the matroid hypergeometric ideal $H_L(u) \subset D_{\mathbb{C}^4}$ is generated by five operators, see \eqref{eq:operatorsbanana} ($n$=3). These operators annihilate the Euler integral
    \begin{equation} \label{eq:banana3} g_\Gamma(z) \, = \, \int_\Gamma \left( -\frac{z_0}{t_0+t_1+t_2} + \frac{z_1}{t_0} + \frac{z_2}{t_1} + \frac{z_3}{t_2}\right)^{-s} \, (-t_0-t_1-t_2)^{u_0}t_0^{u_1}t_1^{u_2}t_2^{u_3} \, \Omega(t_0,t_1,t_2) ,\end{equation}
    which uses $A$ from \eqref{eq:generic matroid}.
    For an appropriate choice of the twisted cycle $\Gamma$, we have the identity
    \[ g_\Gamma(z) \, = \, (-1)^{u_0}(-z_0)^{-s}F_C^{(3)}\left(
    \begin{matrix} s,1+s+u_0 \\ -u_1,-u_2,-u_3 \end{matrix} \, ; \, -\frac{z_1}{z_0},-\frac{z_2}{z_0},-\frac{z_3}{z_0}\right), \]
    where $F_C^{(n)}$ is Lauricella's hypergeometric series of type $F_C$ in $n$ variables \cite[Chapter~2]{Exton1976}:
    \[ F_C^{(n)}\left(\substack{a,b\\ c_1,\dots,c_n};y_1,\dots,{y_n}\right) \, = \, \sum_{m_1,\dots,m_n=0}^{\infty}
    \frac{(a)_{m_1+\dots+m_n}\,(b)_{m_1+\dots+m_n}}
         {(c_1)_{m_1}\,\dots\,(c_n)_{m_n}}
    \frac{y_1^{m_1}}{m_1!}\,\cdots
    \frac{y_n^{m_n}}{m_n!}.  \]
    Here $(a)_m:=a\cdot (a+1)\cdots(a+m-1)$ is the Pochhammer symbol.
    The annihilator ideal of $F_C^{(3)}$ in the Weyl algebra $D = \mathbb{C} \langle y_1, y_2, y_3, \frac{\partial}{\partial y_1}, \frac{\partial}{\partial y_2}, \frac{\partial}{\partial y_2} \rangle$ is holonomic.
The defining equation $E(y_1,y_2,y_3) = 0$ of the singular locus of the corresponding $D$-module is found explicitly in \cite[Theorem 11]{hattori2014singular}.
The hypersurface defined by the numerator of $z_0E(-\frac{z_1}{z_0},-\frac{z_2}{z_0},-\frac{z_3}{z_0})$ is the vanishing locus of $E_L$. This hypersurface is the singular locus of the matroid hypergeometric system $M_L(u)$, see Example \ref{ex:M2}.
The specialization $(a,b,c_1,c_2,c_3)=(s,1+s+u_0,-u_1,-u_2,-u_3)$ gives a reducible monodromy representation as in \cite[Theorem 13]{hattori2014singular}.
\end{example}

As it turns out, integrals like \eqref{eq:eulerint_intro} appear as Feynman integrals of banana diagrams in particle physics (Example \ref{ex:banana integral}). In the language of the physics literature, Theorem \ref{thm:mainintro2} states that the reciprocal matroid determinant contains the Landau singularities of such Feynman integrals. Our matroid hypergeometric system generalizes the differential operators in \cite[Theorem 1]{Flieger} for banana Feynman integrals.

The outline of our paper is as follows. In Section \ref{sec:2}, we recall the $A$-theory of Gelfand, Kapranov and Zelevinsky with the aim of making the analogies with the results of this paper as explicit as possible. Parts 1 and 2 in Theorem \ref{thm:mainintro} will be proved in Section \ref{sec:3}. After that, in Section \ref{sec:4}, we use tropical geometry to investigate the dimension, degree and Newton polytope of the matroid discriminant $\nabla(L^{-1})$. We show that when $\nabla(L^{-1})$ is a hypersurface, its Newton polytope is a simplex (Theorem \ref{thm:maintropical}) and characterize when $\nabla(L^{-1})$ is a hypersurface (Proposition \ref{prop:whenhypersurface}). These results imply Theorem \ref{thm:mainintro}(3). In Section \ref{sec:5} we introduce the matroid hypergeometric system $M_L(u)$ and we prove parts 1 and 2 of Theorem \ref{thm:mainintro2}. In Section \ref{sec:6}, we make the connection with Euler discriminants and prove part 4 of Theorem \ref{thm:mainintro}. In Section 7, we suggest a list of open problems and ideas for future research.

\section{Preliminaries on GKZ theory} \label{sec:2}

Let $X^\circ \subset \mathbb{P}^n$ be an irreducible, smooth, quasi-projective subvariety of dimension $d$.
For a point $z = (z_0 : \cdots : z_n) \in (\mathbb{P}^n)^\vee$ in the dual projective space, we write $H_z = \{ x \in \mathbb{P}^n \, : \, \sum_{i=0}^n z_i \, x_i = 0 \} \subset \mathbb{P}^n$. The \emph{conormal bundle} of $X^\circ$ is defined as follows:
\[ {\rm Con}(X^\circ) \, = \, \{ (x, z) \in \mathbb{P}^n \times (\mathbb{P}^n)^\vee \, : \, x \in X^\circ \text{ and } H_z \supseteq T_xX^\circ \}, \]
where $T_x X^\circ \simeq \mathbb{P}^d$ is the $d$-dimensional projective tangent space of $X^\circ$ at $x$. Explicitly, if $X$ is locally defined by homogeneous equations $f_1(x) = \cdots = f_\ell(x) = 0$ at $x^* \in X$, then
\[ T_{x^*} X^\circ \, = \, \{ x \in \mathbb{P}^n \, : \, J_{X^\circ}(x^*) \cdot x \, = \, 0 \}, \quad \text{where} \quad J_{X^\circ}(x) \, = \, \Big(  \frac{\partial f_i}{\partial x_j} \Big )_{i = 1,\ldots, \ell, \, j = 0, \ldots, n}.  \]
In particular, by Euler's relation, we have $x^* \in T_{x^*}X^\circ$. 
The \emph{conormal variety} of $X^\circ$ is the closure of ${\rm Con}(X^\circ)$ in $\mathbb{P}^n \times (\mathbb{P}^n)^\vee$. The image of ${\rm Con}(X^\circ)$ under the natural projection ${\rm pr}_1: \mathbb{P}^n \times (\mathbb{P}^n)^\vee \rightarrow \mathbb{P}^n$ is $X^\circ$. The other projection ${\rm pr}_2: \mathbb{P}^n \times (\mathbb{P}^n)^\vee \rightarrow (\mathbb{P}^n)^\vee$ yields a geometric definition of discriminants. This definition is standard, see \cite[page 15]{GKZbook}.
\begin{definition} \label{def:disc}
    Let $X \subset \mathbb{P}^n$ be an irreducible closed subvariety. The \emph{$X$-discriminant variety} or \emph{projective dual variety} of $X$ is 
    \[ \nabla(X) \, = \, \overline{{\rm pr}_2({\rm Con}(X_{\rm reg}))} \, \subset \, (\mathbb{P}^n)^\vee,\]
    where $X_{\rm reg} = X \setminus {\rm Sing(X)}$ consists of the regular points of $X$. If $\nabla(X) \subset (\mathbb{P}^n)^\vee$ is a hypersurface, then its defining equation is denoted by $\Delta(X) \in \mathbb{C}[z_0, \ldots, z_n]$. If $\nabla(X)$ has codimension greater than one, then we set $\Delta(X) = 1$. We call $\Delta(X)$ the \emph{$X$-discriminant}.
\end{definition}

Another standard construction we will need is the \emph{Chow form}. 
Let $\mathbb{G}(n-d-1,n)$ be the Grassmannian of $(n-d-1)$-dimensional linear spaces in $\mathbb{P}^n$. We represent a point $[\Lambda] \in \mathbb{G}(n-d-1,n)$ by a $(d+1) \times (n+1)$ matrix $U$ such that $\Lambda = \Lambda_U = \mathbb{P} \ker U \, = \,  \{ x \in \mathbb{P}^n \, : \, U \cdot x = 0 \}$. The maximal minors of $U$ are the \emph{primal Pl\"ucker coordinates} of $[\Lambda_U]$, see \cite{DalbecSturmfels1995}. 
\begin{definition} \label{def:chowform}
    The \emph{Chow variety} of an irreducible variety $X \subset \mathbb{P}^n$ of dimension $d$ is
    \[ {\cal C}(X) \, = \, \{[\Lambda] \in \mathbb{G}(n-d-1,n) \, : \, X \cap \Lambda \neq \emptyset \}.\]
    Its defining equation in primal Pl\"ucker coordinates on $\mathbb{G}(n-d-1,n)$ is the \emph{Chow form} of $X$. It is denoted by ${\rm Chow}(X)$.
\end{definition}
Definition \ref{def:chowform} is justified by the fact that the Chow variety of $X$ is an irreducible hypersurface in $\mathbb{G}(n-d-1,n)$ and, modulo the Pl\"ucker relations, any hypersurface in $\mathbb{G}(n-d-1,n)$ is defined by a single polynomial in the primal Pl\"ucker coordinates. For more on Chow forms, the reader can consult \cite[Chapter 3, Section 2]{GKZbook} or the survey article \cite{DalbecSturmfels1995}.

We recall the definition of the $X$-resultant. Let $p: {\rm Mat}_{d+1, n+1}(\mathbb{C}) \dashrightarrow {\mathbb{G}}(n-d-1,n)$ be the surjective rational map which sends $U \mapsto [\Lambda_U]$. 

\begin{definition} \label{def:res}
    The \emph{$X$-resultant variety} is the hypersurface in ${\rm Mat}_{d+1,n+1}(\mathbb{C})$ given by 
    \[ {\cal R}(X) \, = \, \overline{p^{-1}({\cal C}(X))}. \]
    Its defining equation ${\rm Res}(X) = p^*{\rm Chow}(X) \in \mathbb{C}[u_{ij}]$ is the \emph{$X$-resultant}. 
\end{definition}

In Definition \ref{def:res}, the ring $\mathbb{C}[u_{ij}]$ is the polynomial ring with variables $u_{ij}, \, i = 0, \ldots, d, \, j = 0, \ldots, n$ representing the entries of $U \in {\rm Mat}_{d+1,n+1}(\mathbb{C})$, see Example \ref{ex:introcontd}. 

A main focus of the book \cite{GKZbook} is to apply these general constructions to a projective toric variety $X_A \subset \mathbb{P}^d$ obtained from a configuration $A$ of $n+1$ exponent vectors in $\mathbb{Z}^{d+1}$. The rest of this section recalls standard facts from this $A$-theory of discriminants and resultants, so as to highlight analogies (and differences) with our new results proved in later sections. 

Let $A$ be an integer $(d+1) \times (n+1)$ matrix whose columns are the exponent vectors used in a monomial map $\Phi_A: (\mathbb{C}^\times)^{d+1} \rightarrow \mathbb{P}^n$. The closure of the image of $\Phi_A$ is $X_A$. We assume that the first row of $A$ is the all-ones vector of length $n+1$, which ensures that the prime vanishing ideal $I(X_A)$ of our toric variety is the lattice ideal $I_A$ of the matrix $A$. Moreover, we assume that $A$ has rank $d+1$, so that $X_A$ has dimension $d$. The \emph{$A$-discriminant variety} is $\nabla(X_A)$, and the \emph{$A$-discriminant} is $\Delta(X_A)$ (up to scaling). Similarly, the \emph{$A$-resultant variety} and the \emph{$A$-resultant} are defined as ${\cal R}(X_A)$ and ${\rm Res}(X_A)$ respectively.

An equivalent definition of ${\cal R}(X_A)$ is as follows: ${\cal R}(X_A)$ is the closure in ${\rm Mat}_{d+1,n+1}(\mathbb{C})$ of all coefficients $u_{ij}$ for which the following equations: 
\begin{equation} \label{eq:ueqs} \sum_{j = 0}^n u_{ij} \, t^{a_j} \, = \, 0, \quad i = 0, \ldots, d\end{equation}
have a solution in $(\mathbb{C}^\times)^{d+1}$, where $t^{a_j} = t_0^{a_{0j}}t_1^{a_{1j}} \cdots t_d^{a_{dj}}$. 
Thus, ${\cal R}(X)$ is a generalization of the classical resultants studied by Cayley \cite{cayley1848elimination}. 

To define the principal $A$-determinant, we consider an $(n+1)$-dimensional linear subspace of matrices in ${\rm Mat}_{d+1,n+1}(\mathbb{C})$. For $z \in \mathbb{C}^{n+1}$, we define a matrix $A_z\in {\rm Mat}_{d+1,n+1}(\mathbb{C})$ by  
\begin{equation}  \label{eq:Az} A_z \, = \, A \cdot {\rm diag}(z_0, \ldots, z_n) \, = \, \begin{pmatrix}
    z_0 \, a_{00} & z_1 \, a_{01} & \cdots & z_n \, a_{0n} \\ 
    z_0 \, a_{10} & z_1 \, a_{11} & \cdots & z_n \, a_{1n} \\ 
    \vdots & \vdots & \ddots & \vdots \\ 
    z_0 \, a_{d0} & z_1 \, a_{d1} & \cdots & z_n \, a_{dn}
 \end{pmatrix}.
 \end{equation}
 We define a map $\iota: \mathbb{C}^{n+1} \rightarrow {\rm Mat}_{d+1,n+1}(\mathbb{C})$ by $z \mapsto A_z$. The \emph{principal $A$-determinant} $E_A$ is a polynomial in the variables $z_0, \ldots, z_n$ obtained by subsituting $U = A_z$ in the $A$-resultant ${\rm Res}(X_A)$. That is, we have $E_A = \iota^* {\rm Res}(X_A)$.
The equations \eqref{eq:ueqs} for $U = A_z$ are 
\[ t_0 \frac{\partial f(t,z)}{\partial t_0} \, = \, t_1 \frac{\partial f(t,z)}{\partial t_1} \, = \, \cdots \, = \,  t_d \frac{\partial f(t,z)}{\partial t_d} \, = \, 0,  \]
where $f(t,z) = \sum_{j = 0}^n z_j \, t^{a_j}$. Hence, the principal $A$-determinant is the evaluation of the $A$-resultant at the critical equations of $f(t,z)$. 

The convex hull of the columns of $A$, viewed as points in $\mathbb{R}^{d+1}$, is a $d$-dimensional convex polytope $P(A)$. For each face $Q$ of $P(A)$, we denote by $A_Q$ the submatrix of $A$ consisting of columns representing points on $Q$. The $A_Q$-discriminant $\Delta(A_Q)$ corresponding to that submatrix is expressed in the variables $z_i$ for which the $i$-th column of $A$ lies on $Q$.
With this notation, the following statement about $E_A$ mirrors Theorem \ref{thm:mainintro} about $\Erec$.

\begin{theorem} \label{thm:mainEA}
    The principal $A$-determinant $E_A$ has the following properties. 
    \begin{enumerate}
        \item $E_A$ has degree $(d+1) {\rm Vol}(P_A)$, where ${\rm Vol}(P_A) = \deg(X_A)$ is the lattice volume of $P_A$. 
        \item The radical of $E_A$ is the product $\prod_{Q \preceq P_A} \Delta(X_{A_Q})$, where $Q$ ranges over all faces of $P_A$. 
        \item The Newton polytope of $E_A$ is the secondary polytope of $A$. 
        \item The zero locus of $E_A$ consists of all points $z \in (\mathbb{P}^n)^\vee$ such that the topological Euler characteristic of the hyperplane section $\{ x \in X_A \, : \, z_0 x_0 + \cdots + z_n x_n = 0 \}$ differs from its generic value. 
    \end{enumerate}
\end{theorem}
\begin{proof}
    For point 1, since $E_A$ is obtained as a linear substitution of the $A$-resultant, it has the same degree as ${\rm Res}(X_A)$. The degree of ${\rm Res}(X_A)$ is $(d+1) \deg(X_A)$, see \cite[Chapter 8, Corollary 2.2]{GKZbook} and the preceding discussion. Points 2 and 3 follow from \cite[Chapter 10, Theorems 1.2 and 1.4]{GKZbook}. Point 4 was shown in \cite[Theorem 1.8]{esterov2013discriminant}.
\end{proof}

We recall a parametric description of the $A$-discriminant $\nabla(X_A)$, known as the \emph{Horn-Kapranov uniformization} \cite{Kapranov1991LogGauss}. 
Let $B \in \mathbb{Z}^{(n+1) \times (n-d)}$ be a Gale dual matrix of $A$, i.e., $A \cdot B = 0$ and $B$ has rank $n-d$.
The rows of $B$ are denoted by $b_0, \ldots, b_n \in \mathbb{Z}^{n-d}$.

\begin{proposition} \label{prop:HornKapranov}
    The $A$-discriminant variety $\nabla(X_A)$ is the closure of the image of the map 
    \[ h_A \, : \, (\mathbb{C}^\times)^{d+1} \times \mathbb{P}^{n-d-1} \, \longrightarrow \, (\mathbb{P}^n)^\vee, \quad (t, u) \, \longmapsto \, (t^{-a_0} \, (b_0 \cdot u): \cdots : t^{-a_n} \, (b_n \cdot u) ). \]
\end{proposition}
\begin{proof}
    The image of $\{t\} \times \mathbb{P}^{n-d-1}$ under the map $h_A$ consist of all hyperplanes $H_z$ which are tangent to $X_A$ at the smooth point $\Phi_A(t) \in X_A$. 
\end{proof}
Using the notation $a \star b = (a_0b_0: \cdots : a_nb_n)$ for $a, b \in T \subset (\mathbb{P}^n)^\vee$, we can conveniently express the Horn-Kapranov uniformization as $h_A(t,u) = \Phi_A(t^{-1}) \star(B\, u)$.

Finally, we recall the definition and some properties of the \emph{$A$-hypergeometric system} of Gelfand--Kapranov--Zelevinsky, also known as the \emph{GKZ system}. This is a holonomic system of partial differential equations whose singular locus, in the sense of $D$-module theory, is the vanishing locus of the principal $A$-determinant.  
In what follows, we only include the most relevant results in view of the main theorems of Section \ref{sec:5} of this paper. 
More details, references and a friendly introduction can be found in the recent survey  \cite{ReicheltSchulzeSevenheckWalther2021}.

Let $c \in \mathbb{C}^{d+1}$ be a vector of $d+1$ complex parameters. The \emph{GKZ ideal} $H_A(c)$ is a left ideal of the Weyl algebra $D_{\C^{n+1}}=\C\langle z_0,\dots,z_n,\partial_0,\dots,\partial_n\rangle$ in the variables $z_i$ and the partial derivatives $\partial_i= \frac{\partial}{\partial z_i}$ generated by the following operators:
\begin{align}
    E_i &\, = \, \sum_{j=0}^n a_{ij} \, z_j \, \partial_j+c_i \quad i = 0, \ldots, d, \label{eq:Euler op}\\ 
    P_h & \, = \,  h(\partial_0, \ldots, \partial_n), \quad h(x_0, \ldots, x_n) \in I_A.\label{eq:toric op} 
\end{align}
The GKZ system is the left $D_{\C^{n+1}}$-module $M_A(c):=D_{\C^{n+1}}/H_A(c)$.
A $D_{\C^{n+1}}$-module is said to be \emph{holonomic} if its  \emph{characteristic variety} is purely $(n+1)$-dimensional. The characteristic variety of $M_A(c)$ lives in the cotangent bundle $T^*\mathbb{C}^{n+1}={\rm Spec}\,\C[x_0,\dots,x_n,z_0,\dots,z_n]$ and is denoted by ${\rm Char}(M_A(c))$. 
It is the vanishing locus of the \emph{characteristic ideal} $I_{\rm char}(M_A(c))\subseteq \mathbb{C}[x_0,\ldots, x_n,z_0, \ldots, z_n]$ obtained as an initial ideal of $H_A(c)$, see \cite[Definition 1.9]{SattelbergerSturmfels2025}. 
The \emph{singular locus} ${\rm Sing}(M_A(c))$ of $M_A(c)$ is a subvariety of $\C^{n+1}$ defined by 
\[ \big ( I_{\rm char}(M_A(c)) : \langle x_0, \ldots, x_n\rangle^\infty \big ) \, \cap \, \mathbb{C}[z_0, \ldots, z_n].\]
This variety contains all singularities, i.e., branch points and poles, of holomorphic functions satisfiying our PDEs.  The singular locus connects the GKZ system with the principal $A$-determinant. 
The following theorem collects \cite[Theorem 3.9]{adolphson1994hypergeometric} and \cite[Corollary~4.17]{Schulze_Walther}.

\begin{theorem}\label{thm:GKZsystem}
The GKZ system $M_A(c)$ has the following properties.
\begin{enumerate}
    \item $M_A(c)$ is a holonomic $D_{\C^{n+1}}$-module.
    \item For generic parameters $c \in \mathbb{C}^{n+1}$, the singular locus of $M_A(c)$ is the zero locus of $E_A$. 
\end{enumerate}
\end{theorem}

We now introduce integral representations of the~solutions to a GKZ system.
Let $f(t, z) =\sum_{j=0}^n z_j \, t^{a_j}$ be as above.
Given $c\in\C^{d+1}$, consider the following \emph{Euler integral}~\cite{GKZ,Matsubara-Heo2023Four}: 
\begin{equation}\label{eq:Euler integral}
    \int_{\Gamma}f(t_0 = 1, t_1, \ldots, t_d,z)^{-c_0} \, t_1^{c_1}\cdots t_d^{c_d} \, \frac{\rm d t_1}{t_1} \wedge \cdots \wedge \frac{\rm d t_d}{t_d}
\end{equation}
Here, $\Gamma$ is an appropriate homology cycle with coefficients in a local system, commonly referred to as a \emph{twisted cycle}, see \cite[Section 3]{Matsubara-Heo2023Four}. 
It is readily seen that the Euler integral \eqref{eq:Euler integral} is annihilated by 
\eqref{eq:Euler op} and \eqref{eq:toric op}.
An elementary proof of this fact is found in \cite[Proposition 4.9]{Matsubara-Heo2023Four}. 
In particular, the singularities of the integral \eqref{eq:Euler integral} are contained in the vanishing locus of $E_A$.
The generalized Euler integrals exhaust the solution space outside of the singular locus when the parameter $c$ is non-resonant, a genericity condition described combinatorially in \cite[Theorem 2.10]{GKZ}.
A $D$-module counterpart of this result is \cite[Corollary 3.8 (3)]{schulze2009hypergeometric}.

\section{Discriminants, principal determinants and matroids} \label{sec:3}

In this section, we shall apply the constructions from Definitions \ref{def:disc}-\ref{def:res} to the varieties $L^{-1}$ and $L^{-2}$ from the Introduction. We make the following assumption throughout the text: 
\begin{equation} \label{eq:assum_noloops} \text{\emph{$L \subset \mathbb{P}^n$ has dimension $d< n$ and is not contained in a coordinate hyperplane},}\end{equation} 
so that $L^{-1} = \overline{\phi_{-1}(L \cap T)}$ is not the empty set. Equivalently, the matroid $M(L)$ has no loops. We introduce the following terminology for $\nabla(L^{-1})$ and $\Delta(L^{-1})$. 

\begin{definition} \label{def:matroiddiscriminant}
    Let $L^{-1}$ be the reciprocal linear space associated with a linear space $L \subset \mathbb{P}^n$. The \emph{matroid discriminant variety} of $L$ is $\nabla(L^{-1})$, and its \emph{matroid discriminant} is $\Delta(L^{-1})$. 
\end{definition}

Fix a matrix $A = (a_{ij}) \in {\rm Mat}_{d+1,n+1}(\mathbb{C})$ whose projectivized row span is the $d$-dimensional linear space $L \subset \mathbb{P}^n$. In particular, $A$ has rank $d+1$. Let $A_z$ be as in \eqref{eq:Az} and recall that $\iota$ is the map $\iota: \mathbb{C}^{n+1} \rightarrow {\rm Mat}_{d+1,n+1}(\mathbb{C}), \, z \mapsto A_z$. 

\begin{definition} \label{def:principalmatroiddeterminant}
The \emph{principal matroid determinant} of $L$ is $\Erec = \iota^*{\rm Res}(L^{-2}) \in \mathbb{C}[z_0, \ldots, z_n]$. 
\end{definition}

Concretely, $\Erec$ is obtained from ${\rm Res}(L^{-2})$ by substituting the entries of $U$ by the entries of $A_z$. Note that $E_L$ is defined up to a nonzero constant factor; a different choice of the matrix $A$ results in a scaling of $E_L$. 
We ignore this dependence of $E_L$ on $A$.

\begin{remark} 
    One sees from Definition \ref{def:principalmatroiddeterminant} that the Chow forms and resultants which are important in our principal matroid construction are those of $L^{-2}$, rather than $L^{-1}$. It would be sensible to emphasize their importance by calling ${\rm Res}(L^{-2}) = p^*{\rm Chow}(L^{-2})$ the \emph{matroid resultant} of $L$, in analogy with Definition \ref{def:matroiddiscriminant}.
    There is a risk for confusion with ${\rm Res}(L^{-1}) = p^*{\rm Chow}(L^{-1})$. The Chow form ${\rm Chow}(L^{-1})$ was studied~in~\cite{KummerVinzant2019}. 
\end{remark}

The columns of the matrix $A$ represent linear forms $\ell_0, \ldots, \ell_n \in \mathbb{C}[t_0, \ldots, t_d]_1$: $\ell_i(t) = a_{0i} \, t_0 + \cdots + a_{di} \, t_d$. These form an arrangement of hyperplanes, denoted by ${\cal A} = V(\ell_0 \cdots \ell_n) \subset \mathbb{P}^d$. 
The resultant variety ${\cal R}(L^{-2})$ is the closure of all $U \in {\rm Mat}_{d+1,n+1}(\mathbb{C})$ for which
\[ \sum_{j = 0}^n \frac{u_{ij}}{\ell_j(t)^2} \, = \, 0, \quad i = 0, \ldots, d \]
has a solution in $\mathbb{P}^d \setminus {\cal A}$. This is the analog of \eqref{eq:ueqs}. Specializing by setting $U = A_z$, we obtain 
\[ \frac{\partial f(t,z)}{\partial t_0} \, = \, \frac{\partial f(t,z)}{\partial t_1} \, = \, \cdots \, =\, \frac{\partial f(t,z)}{\partial t_d} \, = \, 0\]
where $f(t,z)$ is the rational function from the Introduction: 
\[ f(t,z) \, = \, \frac{z_0}{\ell_0(t)} + \frac{z_1}{\ell_1(t)} + \cdots + \frac{z_n}{\ell_n(t)}.\]
Hence, much like the principal $A$-determinant, the principal matroid determinant is the evaluation of the resultant of $L^{-2}$ at the critical equations of the rational function $f(t,z)$.

To compute the degree of $E_L$, we need the following lemma, which involves the \emph{M\"obius invariant} $\mu(L)$ of the matroid $M(L)$ represented by $L$. 
Recall that $\mu(L)$ is the absolute value of the constant term of the characteristic polynomial of $M(L)$. 

\begin{lemma} \label{lem:degLk}
    Let $c(L)$ be the number of connected components of $M(L)$. 
    We have
    \[ \deg L^k \, = \, k^{d-c(L)+1} \quad \text{if \, \, $k >0$} \quad \text{ and} \quad \deg L^k \, = \, (-k)^{d-c(L)+1}\mu(L) \quad \text{if \,\,$k< 0$}.\]
\end{lemma}
\begin{proof}
We deduce these formulae from \cite[Theorems 2.3 and 2.6]{DeyGorlachKaihnsa2020}. The statement for $k >0$ is \cite[Theorem 2.6]{DeyGorlachKaihnsa2020}. For $k < 0$, note that $L^k = (L^{-1})^{-k}$. By \cite[Theorem 2.3]{DeyGorlachKaihnsa2020}, we have
\[ \deg L^k \, = \, \frac{|{\rm Fix}_{-k}(L^{-1})|}{|{\rm Stab}_{-k}(L^{-1})|} \, (-k)^d \, \deg(L^{-1}).\]
 Let us write $r = -k$. The sets ${\rm Fix}_{r}(L^{-1})$ and ${\rm Stab}_{r}(L^{-1})$ are defined as follows. Let $G_r = \{ \zeta \in \mathbb{C} \, : \, \zeta^r = 1 \}$ be the group of $r$-th roots of unity and let ${\cal G}_r = G_r^{n+1}/\{(\zeta, \zeta, \ldots, \zeta) \, : \, \zeta^r = 1 \}$. This group acts on $\mathbb{P}^n$ by coordinatewise multiplication. Let 
\[ {\rm Stab}_{r}(L^{-1}) \, = \, \{g \in {\cal G}_r \, : \, g \cdot L^{-1} = L^{-1} \}, \quad {\rm Fix}_r(L^{-1}) \, = \, \{g \in {\cal G}_r \, : \, g_{|L^{-1}} = {\rm id}_{|L^{-1}} \}.\]
For a generic point $p \in L^{-1}$, we have ${\rm Fix}_r(L^{-1}) = \{ g \in {\cal G}_r \, : \, g \cdot p = p \}$. Since $L^{-1}$ is not contained in a coordinate hyperplane, we have ${\rm Fix}_r(L^{-1}) = \{(1, 1, \ldots, 1)\}$. In particular, $|{\rm Fix}_r(L^{-1})| = 1$. For the set ${\rm Stab}_{r}(L^{-1})$, note that $g \cdot L^{-1} = L^{-1}$ if and only if $g\cdot L = L$. Therefore, by the proof of \cite[Theorem 2.6]{DeyGorlachKaihnsa2020}, we have $|{\rm Stab}_{r}(L^{-1})| = r^{c(L)-1}$. The formula $\deg (L^{-1}) = \mu(L)$ deduced from \cite[Lemma 2]{proudfoot2006broken} concludes the proof.
\end{proof}

We now have all the necessary tools in hand to prove a degree formula for $\Erec$. 

\begin{theorem} \label{thm:degErec}
    Let $L \subset \mathbb{P}^n$ be a $d$-dimensional linear space satisfying the assumption \eqref{eq:assum_noloops}. The principal matroid determinant $E_L$ has degree $(d+1) \,  2^{d -c(L)+1} \, \mu(L)$,
    where $c(L)$ is the number of connected components of the matroid $M(L)$, and $\mu(L)$ is the M\"obius invariant. 
\end{theorem}

\begin{proof}
    By \cite[Chapter 3, Proposition 2.2]{GKZbook}, we have $\deg {\rm Chow}(L^{-2}) = \deg L^{-2}$. The primal Pl\"ucker coordinates have degree $d+1$ in the entries $u_{ij}$ of $U$, so that the total degree of the resultant ${\rm Res}(L^{-2})$ is $(d+1) \, \deg(L^{-2})$. Pulling back along the linear map $\iota$ does not change the degree. The theorem now follows from Lemma \ref{lem:degLk}.
\end{proof}

The inverses of the $\ell_i$ parametrize $L^{-1} \cap T$, where $T \simeq (\mathbb{C}^*)^n$ is the dense torus of $\mathbb{P}^n$. More precisely, the map $\ell^{-1}: \mathbb{P}^d \setminus {\cal A} \rightarrow L^{-1} \cap T$, given by 
\begin{equation} \label{eq:ellinv} (t_0: \cdots : t_d) \, \longmapsto \, \left (  \frac{1}{\ell_0(t)}: \cdots : \frac{1}{\ell_n(t)}\right)  \end{equation}
is an isomorphism. In particular, $L^{-1} \cap T$ is smooth, irreducible and of dimension $d$. 

\begin{lemma} \label{lem:recAdisc_geom}
    The matroid discriminant variety $\nabla(L^{-1})$ is the closure of the image of 
    \[ Y^\circ \, = \, \{ (x, z) \in \mathbb{P}^n \times (\mathbb{P}^n)^\vee \, : \, A_z \cdot x = 0 \text{ and } x \in L^{-2} \cap T \} \]
    under the coordinate projection ${\rm pr}_2: \mathbb{P}^n \times (\mathbb{P}^n)^\vee \rightarrow (\mathbb{P}^n)^\vee$.
\end{lemma}
\begin{proof}
    By taking partial derivates of the coordinates of the parametrization $\ell^{-1}: \mathbb{P}^d \setminus {\cal A} \rightarrow L^{-1} \cap T$ we find the following expression for the tangent space $T_{\ell^{-1}(t)} L^{-1}$: 
    \[T_{\ell^{-1}(t)} L^{-1} \, = \, \mathbb{P} {\rm Row} \begin{pmatrix}
        \frac{\partial \ell_j^{-1}}{\partial t_i}
    \end{pmatrix}  \, = \, \mathbb{P} {\rm Row} (A \cdot {\rm diag}(\ell_0(t)^{-2}, \ldots, \ell_n(t)^{-2})). \]
    This is the projectivized row span of the $(d+1) \times (n+1)$ matrix $A_{\ell^{-2}(t)}$ obtained by scaling the $i$-th column of $A$ by $\ell_i(t)^{-2}$ for $i =0, \ldots, n$. 
    Hence, the conormal bundle of $L^{-1} \cap T$ is 
    \begin{equation} \label{eq:conLArec} {\rm Con}(L^{-1} \cap T) \, = \, \{ (x,z) \in \mathbb{P}^n \times (\mathbb{P}^n)^\vee \, : \, A_z \cdot x^2 = 0 \text{ and } x \in L^{-1} \cap T  \}.\end{equation}
    Here we write $x^2 = (x_0^2: \cdots : x_n^2)$. We define the map $\varphi: {\rm Con}(L^{-1} \cap T) \rightarrow Y^\circ$ given by $(x,z) \mapsto (x^2, z)$ and observe that ${\rm pr}_2({\rm Con}(L^{-1} \cap T)) \, = \,  {\rm pr}_2(\varphi({\rm Con}(L^{-1} \cap T))) \, = \, {\rm pr}_2(Y^\circ)$. 
\end{proof}

The geometric description of the matroid discriminant in Lemma \ref{lem:recAdisc_geom} is what motivates us to consider Chow forms and resultants of the ``squared'' reciprocal varieties $L^{-2}$. The inclusion of $\nabla(L^{-1})$ in the zero locus of $\Erec$ is an immediate consequence of Lemma \ref{lem:recAdisc_geom}.

\begin{lemma} \label{lem:ELisprY}
    The zero locus of the principal matroid determinant is the image of 
    \begin{equation} \label{eq:Y} Y \, = \, \{ (x, z) \in \mathbb{P}^n \times (\mathbb{P}^n)^\vee \, : \, A_z \cdot x = 0 \text{ and } x \in L^{-2}  \} \end{equation}
    under the coordinate projection ${\rm pr}_2: \mathbb{P}^n \times (\mathbb{P}^n)^\vee \rightarrow (\mathbb{P}^n)^\vee$. In particular, it contains $\nabla(L^{-1})$. 
\end{lemma}

\begin{proof}
    The resultant hypersurface ${\cal R}(L^{-2}) \subset {\rm Mat}_{d+1, n+1}(\mathbb{C})$ is the coordinate projection of 
    \[ {\cal Y} \, = \, \{ (x, U) \in \mathbb{P}^n \times {\rm Mat}_{d+1, n+1}(\mathbb{C}) \, : \, U \cdot x  = 0 \text{ and } x \in L^{-2} \}\]
    to the space ${\rm Mat}_{d+1, n+1}(\mathbb{C})$.
    In fact, defining ${\rm pr}_U:\mathbb{P}^n\times {\rm Mat}_{d+1, n+1}(\mathbb{C})\to {\rm Mat}_{d+1, n+1}(\mathbb{C})$ as the coordinate projection, the inclusion ${\cal R}(L^{-2}) \subseteq {\rm pr}_U({\cal Y})$ is clear from the definition, and the other inclusion follows from the fact that ${\cal Y}$ is an irreducible hypersurface (it is a bundle over $L^{-2}$). Specializing the resultant polynomial by setting $U = A_z$, as in the definition of $\Erec$, amounts to replacing ${\cal Y}$ by $Y$. Since $Y^\circ \subseteq Y$, we have $\overline{{\rm pr}_2(Y^\circ)} \subseteq {\rm pr}_2(Y)$. This implies the inclusion $\nabla(L^{-1}) \subseteq {\rm pr}_2(Y)$ by Lemma \ref{lem:recAdisc_geom}.
\end{proof}

We have already seen in Example \ref{ex:introcontd} that the inclusion $\overline{{\rm pr}_2(Y^\circ)} \subseteq {\rm pr}_2(Y)$ is strict --- the principal matroid determinant is a reducible polynomial. To identify the other factors of $\Erec$ listed in Theorem \ref{thm:mainintro}(2), we will make use of a well-known stratification of $L^{-1}$. 

A \emph{flat} of the matroid $M(L)$ of rank $r$ is a subset $F \subseteq \{0, \ldots, n\}$ such that the columns of $A$ indexed by $F$ form a maximal submatrix of rank $r$. In particular, $M(L)$ has a unique flat of rank $r = d+1$ given by $F = \{0, \ldots, n\}$. We write ${\cal F}(M(L))$ for the set of flats of $M(L)$. 

\begin{example}
    If $L$ represents the uniform matroid $M(L)$ of rank $d+1$ on $n+1$ elements, then ${\cal F}(M(L))$ contains the empty flat $\emptyset$, the unique rank-$(d+1)$ flat $\{0, \ldots, n\}$, and $\binom{n+1}{r}$ flats of rank $r$ given by all $r$-element subsets of $\{0, \ldots, n \}$. 
 \end{example}

\begin{example} \label{ex:braid1}
    The \emph{braid arrangement} in $\mathbb{P}^3$ is the arrangement of six planes represented~by 
    \[ A' \, = \, \begin{pmatrix}
        1 & 1 & 1 & 0 & 0 & 0 \\ 
        -1 & 0 & 0 & 1 & 1 & 0 \\ 
        0 & -1 & 0 & -1 & 0 & 1 \\
        0 & 0 & -1 & 0 & -1 & -1
    \end{pmatrix}.\]
    This matrix has rank $3$, so below we shall work with its first three rows and write $A \in {\rm Mat}_{3,6}(\mathbb{C})$ for that submatrix. The linear space $L$ is the plane in $\mathbb{P}^5$ spanned by the rows of $A$. We have
    \[ \begin{matrix} {\cal F}(M(L)) \, = \, \{ \emptyset, \, \{0\}, \, \{1\}, \, \{2\}, \, \{3\}, \, \{4\}, \, \{5\}, \, \{0,5\}, \, \{1,4\}, \, \{2,3\}, \\  \quad \quad \quad \quad \quad \, \{0,1,3\}, \, \{0,2,4\}, \, \{1,2,5\}, \, \{3,4,5\}, \, \{0,1,2,3,4,5\} \}.\end{matrix} \]
    There are six flats of rank one, seven of rank two, and one of rank three. 
\end{example}

For any subset $I \subseteq \{ 0, \ldots, n\}$, let $\Lambda_I \subseteq \mathbb{P}^n$ be the coordinate subspace given by $x_i = 0$ for all $i \notin I$. That is, $\Lambda_I \simeq \mathbb{P}^{|I|-1}$. The torus of $\Lambda_I$ is denoted by 
\[ T_I = \{ x \in \Lambda_I \, : \, \prod_{i \in I} x_i \neq 0 \} \simeq (\mathbb{C}^*)^{|I|-1}. \]
These tori stratify $\mathbb{P}^n$. To stratify the subvarieties $L^{-k}$, we only need $T_F$ for flats  $F$ of~$M(L)$. Let $A_F$ be the $(d+1) \times |F|$  submatrix of $A$ whose columns are indexed by $F$. The projectivized row span of $A_F$ is the projection $L_F$ of $L$ to the coordinate subspace $\Lambda_F$.
\begin{proposition} \label{prop:stratify}
    Let $k \in \mathbb{N} \setminus \{0\}$. The variety $L^{-k}$ equals the disjoint union 
    \begin{equation} \label{eq:decompL-k} L^{-k} \, = \, \bigsqcup_{F \in {\cal F}(M(L))} L^{-k} \cap T_F.\end{equation}
    Equivalently, if $I \subseteq \{0, \ldots, n\}$ is not a flat of $M(L)$, then $L^{-k} \cap T_I = \emptyset$. 
    For $F \in {\cal F}(M(L))$, we have that $L^{-k} \cap \Lambda_F \subseteq \Lambda_F \subseteq \mathbb{P}^n$ is $L_F^{-k}$. That is, it equals the closure of the image under $\phi_{-k}$ of $(L \cap T_F) \subset \Lambda_F \simeq \mathbb{P}^{|F|-1}$. Its dimension is $({\rm rank}(F)-1)$. 
\end{proposition}
\begin{proof}
    For $k = 1$, this is \cite[Proposition 5]{proudfoot2006broken}. 
    For $k > 1$, we use the fact that $L^{-k} = \phi_k (L^{-1})$. From this, the decomposition \eqref{eq:decompL-k} follows. Moreover, $\phi_{-k}(L \cap T_F) = (\phi_k \circ \phi_{-1})(L \cap T_F) = \phi_k(L^{-1} \cap T_F)$. 
    Since $\phi_k$ is finite and continuous, the closure of the image of $\phi_k$ is the image of the closure: $\overline{\phi_k(L^{-1} \cap T_F)} = \phi_k(L^{-1} \cap \Lambda_F) = L^{-k} \cap \Lambda_F$.
\end{proof}

\begin{theorem} \label{thm:factorization}
    The zero locus of the principal matroid determinant $\Erec$ is given~by 
    \[ V(\Erec) \, = \, \bigcup_{F \in {\cal F}(M(L))} \nabla(L^{-1} \cap \Lambda_F) \quad \subseteq \quad (\mathbb{P}^{n})^\vee.  \]
\end{theorem}
\begin{proof}
    First, we use Proposition \ref{prop:stratify} to decompose $Y$ from \eqref{eq:Y} as follows: 
    \begin{equation} \label{eq:decompY} Y \, = \, \bigsqcup_{F \in {\cal F}(M(L))} \{ (x, z) \in \mathbb{P}^n \times (\mathbb{P}^n)^\vee \, : \, A_z \cdot x = 0 \text{ and } x \in L^{-2} \cap T_F  \}. \end{equation}
    Let us write $Y_F$ for the component associated with the flat $F$. Let $\varphi: \mathbb{P}^n \times (\mathbb{P}^n)^\vee \rightarrow \mathbb{P}^n \times (\mathbb{P}^n)^\vee$ be the map $(x, z) \mapsto (x^2, z)$ from the proof of Lemma \ref{lem:recAdisc_geom}. We see from \eqref{eq:conLArec} that 
    \[ Y_F \, = \, \varphi( {\rm Con}(L^{-1} \cap T_F)), \]
    so that ${\rm pr}_2(Y_F) = {\rm pr}_2({\rm Con}(L^{-1} \cap T_F))$. Therefore, we have
    \begin{equation}  \label{eq:pr2Y=...} {\rm pr}_2(Y) \, = \, \bigcup_{F \in {\cal F}(M(L))} {\rm pr}_2({\rm Con}(L^{-1} \cap T_F)). \end{equation} Since $L^{-1} \cap T_F$ is dense in $L^{-1} \cap \Lambda_F$, so is ${\rm pr}_2({\rm Con}(L^{-1} \cap T_F))$ in $\nabla(L^{-1} \cap \Lambda_F)$. Taking closures on both sides of \eqref{eq:pr2Y=...} and applying Lemma \ref{lem:ELisprY}, we obtain the statement of the theorem. 
\end{proof}

\begin{corollary} \label{cor:factorization}
    There exist positive integers $m_F$ such that the principal matroid determinant $E_L$ factors as follows: 
    \[ E_L \, = \, \prod_{F \in {\cal F}(M(L))} \Delta(L^{-1} \cap \Lambda_F)^{m_F}.\]
\end{corollary}
Conjecture \ref{conj:multiplicities} predicts the values of the exponents $m_F$ in terms of the geometry of $L^{-2}$.

\begin{example}
    Let $A$ be as in Example \ref{ex:braid1}. The surface $L^{-1} \subset \mathbb{P}^5$ has degree $\mu(L) = 6$. Its ideal is generated by four quadratic trinomials, each associated to one of the circuits of $M(L)$ \cite[Theorem 4]{proudfoot2006broken}. The matroid discriminant of $L$ is an irreducible polynomial $\Delta(L^{-1})$ of degree 8 in $z_0, z_1, \ldots, z_5$. We display a few of its 1185 terms for concreteness: 
    \[ \Delta(L^{-1}) \, = \, z_0^8 + 8 z_0^7 z_1 + 8 z_0^7 z_2 - 8 z_0^7 z_3 - 8 z_0^7 z_4 + 28 z_0^6 z_1^2 + 56 z_0^6 z_1 z_2 - 40 z_0^6 z_1 z_3 - 48 z_0^6 z_1 z_4 + \cdots z_5^8.\]
    The rank-two submatrix corresponding to the flat $F = \{0,1,3 \}$ is 
    \[ A_{\{0,1,3\}} \, = \, \begin{pmatrix}
        1 & 1 & 0 \\ -1 & 0 & 1 \\ 0 & -1 & -1
    \end{pmatrix}.\]
    Its reciprocal row span $L_{\{0,1,3\}}^{-1}$ is a curve in $\mathbb{P}^2$ with coordinates $x_0, x_1, x_3$ given by $x_1x_3 - x_0x_3 + x_0x_1 = 0$. The stratum $L^{-1} \cap T_{\{0,1,3\}}$ is explicitly given by  
    \[  \{ x \in \mathbb{P}^5 \, : \, x_2 = x_4 = x_5 = 0, \,  \, x_1x_3 - x_0x_3 + x_0x_1 = 0, \,  \, x_0x_1x_3 \neq 0 \}. \]
    The reciprocal discriminant of its closure is $\Delta(L_{\{0,1,3\}}^{-1} \cap \Lambda_{\{0,1,3\}}) = z_0^2 + z_1^2 + z_3^2 +z_0z_1 -z_0z_3 + z_1z_3$. 
    The surface $L^{-2} \subset \mathbb{P}^5$ has degree 24, and its ideal is generated by ten quartics. The principal matroid determinant of $L$ has degree $(d+1) 2^d \mu(L) = 3 \cdot 24 = 72$. The resultant in this example is too big to compute. Conjecture \ref{conj:multiplicities} predicts that the polynomial $E_L$ factors as 
    \[ z_0^8 \, z_1^8 \, z_2^8 \, z_3^8 \, z_4^8 \, z_5^8 \, \Delta_{013}^2 \, \Delta_{024}^2 \, \Delta_{125}^2 \, \Delta_{345}^2 \, \Delta(L^{-1}).\]
    Here $\Delta_{013} = \Delta(L_{\{0,1,3\}} \cap \Lambda_{\{0,1,3\}})$ is the polynomial computed above, and similarly for the rank-two flats $\{0,2,4\}, \{1,2,5\}, \{3,4,5\}$. The matroid discriminant varieties corresponding to the other three flats of rank two have codimension two: their reciprocal varieties are lines contained in the coordinate subspaces $\Lambda_{F}$. These discriminant polynomials are equal to 1. 
\end{example}

\section{Parametrization and tropicalization} \label{sec:4}

In this section, we study \emph{matroid discriminants}. By point 2 of Theorem \ref{thm:mainintro}, these are the building blocks of the principal matroid determinant. We continue to assume that the matroid $M(L)$ has no loops, see \eqref{eq:assum_noloops}. We start with a parametrization. 

The Horn-Kapranov uniformization from Proposition \ref{prop:HornKapranov} has a natural analog in our matroid setting. Let $B \in \mathbb{R}^{(n+1) \times (n-d)}$ be a matrix whose columns form a basis for $\ker A$. In particular, the matrix product $A \cdot B$ yields a $(d+1) \times (n-d)$ matrix with zero entries. The $(n-d-1)$-dimensional projectivized column span of $B$ is $L^{\perp}$. 
We define the \emph{Hadamard product} $X \star Y$ of two varieties $X, Y \subseteq \mathbb{P}^n$ to be the Zariski closure of
\[ \{ \, (x_0y_0: \cdots: x_ny_n) \, : \, x \in X, \, y \in Y, \, x_i y_i \text{ are not all zero } \}.\]
The following proposition uses the previously introduced notation $L^2 = \phi_2(L)$, see \eqref{eq:phi_k}. 
\begin{proposition} \label{prop:reciprocalHornKapranov}
The matroid discriminant variety of $L$ is given by $\nabla(L^{-1}) = L^2 \star L^\perp$. 
\end{proposition}
\begin{proof}
The image of the parametrization $\ell^{-1}$ from \eqref{eq:ellinv} is $L^{-1} \cap T$, which is a dense subset of the smooth points of $L^{-1}$. In particular, the points $z \in (\mathbb{P}^n)^\vee$ representing hyperplanes $z_0 x_0 + \cdots + z_n x_n = 0$ which are tangent to ${\rm im} \, \ell^{-1}$ are dense in the dual variety $\nabla(L^{-1})$.

The $(d+1) \times (n+1)$ Jacobi matrix of the map $\ell^{-1}$ is  
\[ J(t) \, = \, \Big ( \frac{\partial (\ell_j^{-1}(t))}{\partial t_i} \Big)_{i = 0, \ldots, d, j = 0, \ldots, n} \, = \, -A \cdot {\rm diag}(\ell_0(t)^{-2}, \ldots, \ell_n(t)^{-2}). \]
A point $z \in (\mathbb{P}^n)^\vee$ represents a hyperplane $z_0 x_0 + \cdots + z_n x_n = 0$ which is tangent to $L^{-1}$ at $\ell^{-1}(t)$ if and only if the vector $(z_0, \ldots, z_n)$ lies in the kernel of $J(t)$. That is, $z$ is of the form $(\ell_0(t)^2 \, v_0, \ldots, \ell_n(t)^2 \, v_n)$ for some $v \in \ker A = L^{\perp}$, concluding the proof.   
\end{proof}

Concretely, Proposition \ref{prop:reciprocalHornKapranov} provides the following parametrization for $\nabla(L^{-1})$:
\begin{equation} \label{eq:hL} h_L \, : \, \mathbb{P}^d \times \mathbb{P}^{n-d-1} \, \longrightarrow \, (\mathbb{P}^n)^\vee, \quad (t,u) \longmapsto ((t \cdot a_0)^2(b_0\cdot u) : \cdots : (t \cdot a_n)^2(b_n\cdot u)),\end{equation}
or $h_L(t,u) = \phi_2(t^\top A) \star (B \,u)$.
Here $b_0, \ldots, b_n \in \mathbb{R}^{n-d}$ are the rows of the matrix $B$. 

The \emph{coloops} of the matroid $M(L)$ are the elements $i \in [n]$ such that the $i$-th row of $B$ is zero. That is, $i$ is a \emph{loop} in the dual matroid $M(L^\perp)$. Here is a consequence of Proposition~\ref{prop:reciprocalHornKapranov}.

\begin{corollary} \label{cor:forgetloops}
Let $J \subset [n] = \{0, \ldots, n\}$ be the set of coloops of $M(L)$. The matroid discriminant variety $\nabla(L^{-1})$ is contained in the coordinate subspace $(\mathbb{P}^{n - |J|})^\vee$ given by $z_i = 0$ for $i \in J$. Writing $L_{\hat{J}} $ for the projection of $L$ to the coordinate subspace labeled by $[n] \setminus J$, we have that $\nabla(L^{-1}) = \nabla(L_{\hat J}^{-1}) \subseteq (\mathbb{P}^{n - |J|})^\vee \subseteq (\mathbb{P}^n)^\vee$.
\end{corollary}

Some statements below assume that $M(L)$ has no coloops. This is not restrictive: by Corollary \ref{cor:forgetloops}, we may always reduce to the coloopless case, replacing $L$ with~$L_{\hat{J}}$. 

Theorem \ref{thm:mainintro} gives a simple formula for the degree of the  principal matroid determinant. Like in the case of toric $A$-discriminants, determining the degree of the matroid discriminants is more complicated. Taking cues from \cite{DickensteinFeichtnerSturmfels2007}, we use the parametrization from Proposition \ref{prop:reciprocalHornKapranov} to study the \emph{tropicalization} of $\nabla(L^{-1})$. This gives, in particular, a combinatorial algorithm to compute its degree, and we shall derive a closed formula in the case where $L$ represents the uniform matroid. We recall some basic facts and introduce more notation. 

The tropicalization of an irreducible $d$-dimensional subvariety $X^\circ \subset T$ is given by 
\[ {\rm Trop}(X^\circ) \, = \, \{w \in \mathbb{R}^{n+1}/ \mathbb{R} {\bf 1} \, : \, {\rm in}_w(I(X^\circ)) \neq \langle 1 \rangle \}. \]
Here ${\rm in}_w(I(X^\circ))$ is the initial ideal with respect to the weight vector $w$ of the defining ideal of $X^\circ$ in the Laurent polynomial ring. For a subvariety $X \subset \mathbb{P}^n$, we define ${\rm Trop}(X) = {\rm Trop}(X^\circ)$, where $X^\circ = X \cap T$. The set ${\rm Trop}(X^\circ)$ can be given a (non-unique) fan structure, such that each maximal cone has dimension $d$. After assigning weights to each $d$-dimensional cone in an appropriate way, our fan is balanced \cite[Corollary 3.5.5]{maclagansturmfels}.

The operations of taking $k$-th powers $X \mapsto \overline{\phi_k(X \cap T)} = X^k$ and taking Hadamard products $X \star Y$ of projective varieties have simple tropical analogs. 

\begin{lemma} \label{lem:tropicalize}
Let $X, Y \subset \mathbb{P}^n$ be irreducible projective varieties satisfying $\overline{X \cap T} = X$ and $\overline{Y \cap T} = Y$. Let $k$ be a positive integer. We have the following equalities of subsets of $\mathbb{R}^{n+1}/\mathbb{R} {\bf 1}$: 
\[ {\rm Trop}(X^k) \, = \, {\rm Trop}(X), \quad {\rm Trop}(X \star Y) \, = \, {\rm Trop}(X) + {\rm Trop}(Y). \]
In the second equality, `$+$' denotes the Minkowski sum in the vector space $\mathbb{R}^{n+1}/ \mathbb{R} {\bf 1}$.
\end{lemma}

\begin{proof}
These are consequences of Corollary 3.2.13 and Proposition 5.5.11 in \cite{maclagansturmfels}.
\end{proof}

\begin{proposition} \label{prop:tropicalization_as_sets}
    As a set, the tropicalization of the matroid discriminant variety is the Minkowski sum of the tropical linear spaces ${\rm Trop}(L)$ and ${\rm Trop}(L^\perp)$. In Symbols, we have ${\rm Trop}(\nabla(L^{-1})) = {\rm Trop}(L) + {\rm Trop}(L^\perp)$. In particular, the dimension of the matroid discriminant only depends on the matroid $M(L)$.
\end{proposition}
\begin{proof}
    This follows immediately from Proposition \ref{prop:reciprocalHornKapranov} and Lemma \ref{lem:tropicalize}. The tropical linear spaces ${\rm Trop}(L)$ and ${\rm Trop}(L^\perp)$ depend only on $M(L)$ \cite[Proposition 4.4.4]{maclagansturmfels}.
\end{proof}

\begin{proposition} \label{prop:reducetoLALB}
If $M(L)$ has no coloops, then $L^k \star (L^\perp)^l$ has the same dimension for any $k, l \in \mathbb{Z}_{>0}$. In particular, the dimension of $\nabla(L^{-1})$ equals the dimension of $L \star L^\perp$. 
\end{proposition}
\begin{proof}
Since $M(L)$ has no loops (by our standing assumption \eqref{eq:assum_noloops}) and no coloops, we have $\overline{L \cap T} = L$ and $\overline{L^\perp\cap T} = L^\perp$. By Lemma \ref{lem:tropicalize}, we have 
\[ {\rm Trop}(L^k  \star (L^\perp)^l) \, = \, {\rm Trop}(L^k ) + {\rm Trop}((L^\perp)^l) \, = \, {\rm Trop}(L \star L^\perp).\]
If the tropicalizations are equal, then so are the dimensions. 
\end{proof}

Proposition \ref{prop:reducetoLALB} motivates us to study the variety $L \star L^\perp$ in some more detail. 

\begin{proposition} \label{prop:LAstarLB}
    For any linear space $L \subsetneq \mathbb{P}^n$, the variety $L \star L^\perp$ is contained in the hyperplane given by the coordinate sum $z_0 + z_1 + \cdots + z_n = 0$. 
\end{proposition}
\begin{proof}
    The variety $L \star L^\perp$ is the closure of the image of 
    \begin{equation} \label{eq:auxjacobian} (t,u) \, \longmapsto \, \big ( (t \cdot a_0) (b_0 \cdot u) : \cdots : (t \cdot a_n)(b_n \cdot u) ), \end{equation}
    with $a_i$ and $b_j$ as in \eqref{eq:hL}. This is a rational map $\mathbb{P}^d \times \mathbb{P}^{n-d-1} \dashrightarrow \mathbb{P}^n$. The sum of the coordinates $z_i$ of the image is a bilinear polynomial in $t,u$. Expanding it gives $\sum_{i = 0}^n z_i = \sum_{k,l} (A \cdot B)_{k,l} \, t_k \, u_{l} \, = \, 0$. Here $(A \cdot B)_{k,l}$ is the $(k, l)$-entry of the matrix product $A \cdot B = 0$.
\end{proof}

\begin{theorem} \label{thm:maintropical}
    If $M(L)$ has no coloops and the matroid discriminant variety $\nabla(L^{-1})$ is a hypersurface, then the Newton polytope of $\Delta(L^{-1})$ is a (dilated) standard simplex. 
\end{theorem}

\begin{proof}
    If $\nabla(L^{-1})$ is a hypersurface, then so is $L \star L^\perp$ by Proposition \ref{prop:reducetoLALB}. Moreover, by Proposition \ref{prop:LAstarLB}, $L \star L^\perp$ is the hyperplane given by $z_0 + z_1 + \cdots + z_n = 0$. In particular, its tropicalization ${\rm Trop}(L \star L^\perp)$ is the normal fan of the standard simplex. Indeed, the tropicalization of a hypersurface in $T$ equals the $(n-1)$-skeleton of the normal fan of its Newton polytope \cite[Proposition 3.1.10]{maclagansturmfels}. By Propositions \ref{prop:reciprocalHornKapranov} and \ref{prop:reducetoLALB}, we have ${\rm Trop}(L \star L^\perp) = {\rm Trop}(\nabla(L^{-1}))$ (as sets, not as weighted polyhedral fans), which proves that the Newton polytope of $\Delta(L^{-1})$ is a dilation of the standard simplex ${\rm Conv}(e_0, \ldots, e_n) \subset \mathbb{R}^{n+1}$.
\end{proof}

\begin{theorem} \label{thm:newtEL}
    Define the subset ${\cal B} = \{ F \in {\cal F}(M(L)) \, : \, \text{$\nabla(L_F^{-1})$ is a hypersurface }\}$.
    The Newton polytope of the principal matroid determinant is the generalized permutohedron
    \begin{equation} \label{eq:newtEL} {\rm Newt}(E_L) \, = \, \sum_{F \in {\cal B}} m_F \, \deg(\Delta(L_F^{-1})) \,  \Delta_F,\end{equation}
    where $\Delta_F$ is the standard simplex ${\rm Conv}(e_i \,: \, i \in F).$
\end{theorem}
\begin{proof}
    The formula \eqref{eq:newtEL} follows directly from Corollary \ref{cor:factorization} and Theorem \ref{thm:maintropical}. This polytope is a generalized permutohedron by the characterization in \cite[Proposition 6.3]{Postnikov2009Permutohedra}.
\end{proof}

We now describe how to use Proposition \ref{prop:reciprocalHornKapranov} to compute ${\rm Trop}(\nabla(L^{-1}))$ as a balanced weighted fan. 
First, we determine the tropicalization of $L^2$ by using \cite[Theorem 3.12]{SturmfelsTevelev2008}.
We compute a simplicial fan $\Sigma_L$ whose support is the tropicalization of $L$. Since ${\rm Trop}(L) = {\rm Trop}(L^2)$ as sets, the support of $\Sigma_L$ also equals the tropicalization of $L^2$. Let $\Sigma_L(d)$ be the set of $d$-dimensional cones of $\Sigma_L$. In ${\rm Trop}(L)$, each $\sigma \in \Sigma_L(d)$ is assigned weight one. What matters for us is the correct assignment of weights for ${\rm Trop}(L^2)$.

\begin{lemma} \label{lem:weightsLA2}
    The tropicalization of $L^2$ equals the support of the fan $\Sigma_L$. The weight of each cone in $\Sigma_L(d)$ is $2^{d-c(L)+1}$, where $c(L)$ is the number of connected components of $M(L)$. 
\end{lemma}
\begin{proof}
    This is a direct consequence of \cite[Theorem 3.12]{SturmfelsTevelev2008} and Lemma \ref{lem:degLk}.
\end{proof}

The tropicalization of $L^\perp$ is the support of a fan $\Sigma_{L^\perp}$ in $\mathbb{R}^{n+1}/\mathbb{R} \mathbf{1}$, and each $(n-d-1)$-dimensional cone in $\Sigma_{L^\perp}(n-d-1)$ is assigned weight one. 

We can now use this information, i.e., the weighted fans $\Sigma_L$ and  $\Sigma_{L^\perp}$, to compute the tropicalization of the Hadamard product $L^2 \star {L^\perp} = \nabla(L^{-1})$. As a set, this was determined in Proposition \ref{prop:tropicalization_as_sets}: ${\rm Trop}(L^2 \star {L^\perp}) = {\rm Trop}(L) + {\rm Trop}({L^\perp})$. It remains to assign the correct weights to the cones of ${\rm Trop}(L^2 \star {L^\perp})$, for which we will use \cite[Theorem 3.12]{SturmfelsTevelev2008}. The rest of the discussion applies in the case where $\nabla(L^{-1})$ is a hypersurface. 

We interpret $L^2 \star {L^\perp}$ as the image of $L^2 \times {L^\perp}$ under the monomial map $(x,y) \mapsto x \star y$.  Let $\Sigma$ be a fan whose support is ${\rm Trop}(L^2 \star {L^\perp})$, such that for each $\sigma_{L} \in \Sigma_L(d)$ and $\sigma_{L^\perp} \in \Sigma_{L^\perp}(n-d-1)$, the Minkowski sum $\sigma_{L} + \sigma_{L^\perp}$ is a union of cones in $\Sigma$. Let $\delta$ be the degree of the map $L^2 \times {L^\perp} \rightarrow L^2 \star {L^\perp}$. By \cite[Theorem 3.12]{SturmfelsTevelev2008}, the weight of $\sigma \in \Sigma(n-1)$ is 
\begin{equation} \label{eq:weights} {\rm weight}(\sigma) \, = \, \frac{1}{\delta}\sum_{\sigma_{L} + \sigma_{L^\perp} \supseteq \sigma} {\rm weight}(\sigma_{L}) \, {\rm weight}(\sigma_{L^\perp}) \, [N_\sigma : N_{\sigma_{L}} + N_{\sigma_{L^\perp}}]. \end{equation}
The sum is over all tuples $(\sigma_{L}, \sigma_{L^\perp})\in \Sigma_L(d) \times \Sigma_{L^\perp}(n-d-1)$ such that $\sigma_{L} + \sigma_{L^\perp} \supseteq \sigma$. Let $N_C$ be the sublattice of $\mathbb{Z}^{n+1}/\mathbb{Z} {\bf 1}$ generated by the lattice points of $C$. The integer $[N_\sigma : N_{\sigma_{L}} + N_{\sigma_{L^\perp}}]$ is the lattice index of $N_{\sigma_{L}} + N_{\sigma_{L^\perp}}$ inside $N_\sigma$. The weight of $\sigma_{L^\perp}$ is one, and by Lemma \ref{lem:weightsLA2} we have ${\rm weight}(\sigma_{L}) = 2^{d-c(L)+1}$. 
The following lemma determines the degree $\delta$ of the Hadamard product map $L^2 \times {L^\perp} \rightarrow L^2 \star {L^\perp}$ when $\nabla(L^\perp)$ is a hypersurface.
\begin{lemma} \label{lem:delta}
    If $\nabla(L^{-1})$ is a hypersurface, then the map $L^2 \times {L^\perp} \rightarrow L^2 \star {L^\perp}$ which sends $(x,y)$ to $x \star y$ is birational. That is, the number $\delta$ in \eqref{eq:weights} is one. 
\end{lemma}
\begin{proof}
    Let $z \in \nabla(L^{-1})$ be a generic point in the image of the Hadamard product map $L^2 \times {L^\perp} \rightarrow L^2 \star {L^\perp}$. By \cite[Chapter 1, Proposition 3.2(b)]{GKZbook}, there is a unique point $x \in L^{-1}$ such that the hyperplane $\{x \in \mathbb{P}^{n-1} \, : \, z_0 x_0 + \cdots + z_n x_n = 0\}$ is tangent to $L^{-1}$ at $x$. Let $(y_0: \cdots: y_n) = (z_0x_0^{-2}: \cdots : z_n x_n^{-2})$. By construction $(\phi_2(x),y)$ is the unique pre-image of $z = \phi_2(x) \star y$ under the Hadamard product map. 
\end{proof}

If $\nabla(L^{-1})$ is a hypersurface, i.e., if $\Sigma$ has dimension $n-1$, then the Newton polytope of $\Delta(L^{-1})$ is a simplex by Theorem \ref{thm:maintropical}. Its vertices are $w \cdot e_i$, $i \in [n]$, where $e_i$ is the $i$-th standard basis vector in $\mathbb{R}^{n+1}$, and $w = {\rm weight}(\sigma)$ is the weight of any $\sigma \in \Sigma(n-1)$. 

 If $M(L)$ is uniform, we can determine a closed formula for the degree of $\Delta(L^{-1})$ by keeping track of the weights of the cones of $\Sigma$ in the computation explained above.

\begin{theorem} \label{thm:uniform}
    If $M(L)$ is the uniform matroid of rank $d+1$ on $n+1$ elements, then the matroid discriminant variety $\nabla(L^{-1}) \subset \mathbb{P}^n$ is a hypersurface of degree $2^d \binom{n-1}{d}$. 
\end{theorem}

\begin{proof}
For $S \subset \mathbb{R}^{n+1}/\mathbb{R} {\bf 1}$, let ${\rm pos}(S)$ be the cone generated by $S$. The Bergman fan structure on ${\rm Trop}(L)$ \cite[Example 4.2.13]{maclagansturmfels} is the fan $\Sigma_L$ whose $d$-dimensional cones are 
\[ \sigma_{L,J} \, = \, {\rm pos}(e_j \, : \, j \in J) \, \subset \, \mathbb{R}^{n+1}/\mathbb{R} {\bf 1} \quad \text{where} \quad J \subset [n], \, |J| = d.\]
Similarly, we define $\Sigma_{L^\perp}$ to be the $(n-d-1)$-dimensional fan whose maximal cones are 
\[ \sigma_{L^\perp,K} \, = \, {\rm pos}(e_k \, : \, k \in K) \, \subset \, \mathbb{R}^{n+1}/\mathbb{R} {\bf 1} \quad \text{where} \quad K \subset [n], \, |K| = n-d-1.\]
The Minkowski sum of the supports of $\Sigma_L$ and $\Sigma_{L^\perp}$ forms the $(n-1)$-skeleton of the normal fan of the standard simplex, i.e., the $(n-1)$-dimensional fan $\Sigma$ whose maximal cones are
\[ \sigma_{L} \, = \, {\rm pos}(e_l \, : \, l \in L) \, \subset \, \mathbb{R}^{n+1}/\mathbb{R} {\bf 1} \quad \text{where} \quad L \subset [n], \, |L| = n-1.\]
By Proposition \ref{prop:tropicalization_as_sets}, this implies that $\nabla(L^{-1})$ is a hypersurface. 
Since the uniform matroid is connected, we obtain by applying Lemma \ref{lem:weightsLA2}, Lemma \ref{lem:delta} and Equation \eqref{eq:weights} that
\[ {\rm weight}(\sigma_L) \, = \, \sum_{J \cup K = L} {\rm weight}(\sigma_{L,J}) \,  {\rm weight}(\sigma_{L^\perp,K}) \, = \, 2^d \, \binom{n-1}{d}. \]
We now apply \cite[Corollary 3.6.16]{maclagansturmfels} to conclude that $\deg(\nabla(L^{-1})) = 2^d \binom{n-1}{d}$.
\end{proof}
We note that a similar strategy for tropicalization appears in \cite[Remark 5.3 and Theorem 6.8]{BocciCarliniKileel2016}, where the authors compute the tropicalization of a Hadamard product of linear spaces.

We now comment on the dimension of the matroid discriminant variety. When $\nabla(L^{-1})$ has codimension greater than one, the reciprocal linear space $L^{-1}$ is called \emph{dual defective}. Using Corollary \ref{cor:forgetloops}, it is easy to construct examples of matrices $A$ for which $M(L)$ has (co-)loops, and for which $L^{-1}$ is dual defective. For matrices with no loops or coloops, the dimension of $\nabla(L^{-1})$ equals the dimension of $L \star {L^\perp}$ (Proposition \ref{prop:reducetoLALB}). This is the rank of the $(n+1) \times (n+1)$ Jacobian matrix of \eqref{eq:auxjacobian} over $\mathbb{C}(t,u)$ minus one. That matrix is given by
\begin{equation} \label{eq:jacobian}
    {\cal J}_L(t,u) \, = \, \begin{pmatrix}
        {\rm diag}(B \, u) \cdot A^\top & {\rm diag}(A^\top t) \cdot B
    \end{pmatrix}. 
\end{equation}
Notice that the vector $(t,-u)^\top = (t_0,\ldots, t_d, -u_0, \ldots, -u_{n-d-1})^\top$ lies in its kernel, so that the rank of ${\cal J}_L(t,u)$ over $\mathbb{C}(t,u)$ is at most $n$. The following criterion for defectivity is immediate. 
\begin{proposition} \label{prop:whenhypersurface}
    If $M(L)$ has no coloops, then $L^{-1}$ is dual defective if and only if all $n \times n$ minors of the matrix ${\cal J}_L(t,u)$ in \eqref{eq:jacobian} are identically zero. Equivalently, $\nabla(L^{-1})$ is a hypersurface if and only if ${\cal J}_L(t^*,u^*)$ has rank $n$ over $\mathbb{C}$ for some $(t^*,u^*) \in \mathbb{P}^d \times \mathbb{P}^{n-d-1}$.
\end{proposition}

The next example illustrates that not all dual defective cases come from (co-)loops.
\begin{example}[$d = 2, n = 5$]
    We consider matrices $A, B$ of the following form: 
    \[ A \, = \, 
\begin{pmatrix}
1 & 0 & 0 & a_{11} & a_{12} & a_{13} \\
0 & 1 & 0 & a_{21} & a_{22} & a_{23} \\
0 & 0 & 1 & a_{31} & a_{32} & a_{33}
\end{pmatrix}, \quad B^\top \, = \, \begin{pmatrix}
- a_{11} & - a_{12} & - a_{13} & 1 & 0 & 0 \\
- a_{21} & - a_{22} & - a_{23} & 0 & 1 & 0 \\
- a_{31} & - a_{32} & - a_{33} & 0 & 0 & 1
\end{pmatrix}.\]
The $6 \times 6$ matrix ${\cal J}_L(t,u)$ has $36$ minors of size $5 \times 5$. The coefficients of all these minors are polynomials in the nine variables $a_{11}, \ldots, a_{33}$. Together with the polynomial $g = a_{11}a_{22}a_{33}y-1$, these generate an ideal $I$ in the ring $\mathbb{Q}[a_{11}, \ldots, a_{33},y]$. The generator $g$ ensures that each $a_{ii}$ is non-zero, so that $M(L)$ has no loops, and no coloops. The minimal primes of $I$ are
\[ \langle g,\; a_{23},\; a_{13},\; a_{32},\; a_{31} \rangle, \quad 
\langle g,\; a_{23},\; a_{32},\; a_{12},\; a_{21} \rangle, \quad 
\langle g,\; a_{13},\; a_{12},\; a_{31},\; a_{21} \rangle. \]
Each of these gives a five-dimensional family of matrices whose reciprocal linear space is dual defective. For instance, the matrix $a_{11}=a_{22}=a_{33} = 1, a_{23} = a_{13} = a_{32} = a_{31} = 0, a_{12} = 3, a_{21} = 2$ lies on the first component. Its matroid discriminant variety $\nabla(L^{-1})$ has codimension two. Its ideal is generated by $z_2 +z_5$ and a quartic polynomial with $35$ terms.
\end{example}

We conjecture that $\nabla(L^{-1})$ has codimension $> 1$ if and only if $M(L)$ is not connected (Conjecture \ref{conj:dualdefectivity}). If $M(L)$ is connected, then Conjecture \ref{conj:degree} predicts the degree of $\nabla(L^{-1})$.

\section{Matroid hypergeometric systems} \label{sec:5}
Given the fact that the principal matroid determinant resembles the principal $A$-determinant in many ways (see Sections \ref{sec:2} and \ref{sec:3}), it is natural to ask what the analog of the GKZ hypergeometric system is in our matroid setting. Here, we provide an answer to this question. We introduce the \emph{matroid hypergeometric system} and prove some fundamental properties.

As before, $L\subset\mathbb{P}^n$ is a $d$-dimensional linear subspace satisfying \eqref{eq:assum_noloops}, represented as the projectivized row span of a matrix $A = (a_{ij})\in \mathbb{C}^{(d+1) \times (n+1)}$. We shall write $\partial_i$ for the partial derivative with respect to $z_i$. We select a parameter vector $u\in\C^{n+1}$ and determine $s\in\C$~via 
\begin{equation}\label{eq:homogeneity}
s + d+1 +u_0 + \cdots + u_n =0.    
\end{equation}
Consider the following sets of operators in the Weyl algebra $D_{\C^{n+1}}=\mathbb{C}\langle z_i,\partial_i \, ; \,  i=0,\dots,n\rangle$: 
\begin{align}
    H &\, = \, \sum_{j = 0}^n z_j \, \partial_j + s,\label{eq:op1} \\
    P_i&\, = \, \sum_{j=0}^n a_{ij} \, z_j \, \partial_j^2 - \sum_{j=0}^n a_{ij}\, u_j\, \partial_j, \quad i = 0, \ldots, d,\label{eq:op2}\\ 
    Q_h & \, = \,  h(\partial_0, \ldots, \partial_n), \quad h(x_0, \ldots, x_n) \in I(L^{-1}).\label{eq:op3} 
\end{align}
The operator $H$ annihilates functions in $z$ which are homogeneous of degree $-s$. Recall that $I(L^{-1})$ is the homogeneous prime ideal defining the reciprocal linear space $L^{-1} \subset \mathbb{P}^n$. 
The infinite set of operators $Q_h$ annihilates a function $g(z)$ if and only if a finite set of homogeneous generators $h_1, \ldots, h_r$ for the ideal $I(L^{-1})$ annihilates $g(z)$.

The system of partial differential equations given by the operators $H, P_i$ and $Q_h$ is reminiscent of {GKZ systems}, the basics of which were recalled in Section \ref{sec:2}. 
We write $H_L(u)$ for the left ideal of $D_{\C^{n+1}}$ generated by $H, P_i$ and $Q_h$ and $M_L(u)$ for the left $D_{\C^{n+1}}$-module $D_{\C^{n+1}}/H_L(u)$. Notice that $H_L(u)$ and $M_L(u)$ are independent of the matrix $A$ we choose to represent the linear space $L$. 
Note that $M_L(u)$ only depends on $u = (u_0, \ldots, u_n)$ by \eqref{eq:homogeneity}.
\begin{definition} \label{def:matroidhypergeometricsystem}
    Let $L \subset \mathbb{P}^n$ be a linear space. The \emph{matroid hypergeometric system} of $L$ is the $D_{\mathbb{C}^{n+1}}$-module $M_L(u)$ defined above. The \emph{matroid hypergeometric ideal} is $H_L(u)$.
\end{definition}

Let $V \subset \mathbb{P}^n\times (\mathbb{P}^n)^\vee$ be a closed subvariety.
Its {\it multicone} is defined to be the vanishing locus of the defining bihomogeneous ideal $I(V)\subset \mathbb{C}[x_0,...,x_n,z_0,...,z_n]$ in the affine space $\mathbb{C}^{n+1}\times (\mathbb{C}^{n+1})^\vee=T^*\C^{n+1}$.
For instance, in Section \ref{sec:2} we defined the conormal variety ${\rm Con}(X)$ as a subvariety of $\mathbb{P}^n\times(\mathbb{P}^n)^\vee$. Here, it is more convenient to work with its multicone, and we will abuse notation by using the same symbol ${\rm Con}(X)$ for the multicone.

\begin{theorem}\label{thm:holonomicity}
For any choice of the parameters $u \in \mathbb{C}^{n+1}$, the matroid hypergeometric system $M_L(u)$ is a holonomic $D_{\C^{n+1}}$-module.
Moreover, its singular locus is contained in the vanishing locus of the principal matroid determinant $E_L$.
\end{theorem}

\begin{proof}
Let $I\subset\C[x_0,...,x_n,z_0,...,z_n]$ be the defining ideal of the characteristic variety ${\rm Char}(M_L(u))$. By definition, this is the initial ideal of $H_L(u)$ obtained by taking the highest order terms in the symbols $\partial_0, \ldots, \partial_n$ of any operator in $H_L(u)$ and substituting $\partial_i$ by $x_i$.
It follows from the definition of $H_L(u)$ that $\sum_{i=0}^nz_ix_i \in I$, $\sum_{j=0}^na_{ij}z_jx_j^2 \in I$ for each $i=0,\dots,d$ and $h(x_0, \ldots, x_n) \in I$ for each $h \in I(L^{-1})$. 
This implies that ${\rm Char}(M_L(u))$ is contained in 
\[ W \, = \, \hat{Y} \cup V(x_0, \ldots, x_n), \]
where $\hat{Y}$ is 
the multicone of $Y$ from \eqref{eq:Y}. Equations \eqref{eq:conLArec} and \eqref{eq:decompY} give the decomposition 
\begin{equation} \label{eq:W} W \, = \, \bigsqcup_{F\in\mathcal{F}(M(L))}{\rm Con}(L^{-1}\cap T_F). \end{equation}
Here we use the convention that, for the empty flat $F = \emptyset$, ${\rm Con}(L^{-1} \cap T_\emptyset) = V(x_0, \ldots, x_n)$. In $D$-module theory, this is often called the \emph{zero section}. 
    The inclusion ${\rm Char}(M_L(u)) \subseteq W$ shows that ${\rm Char}(M_L(u))$ has dimension $n+1$. Hence, $M_L(u)$ is a holonomic $D_{\C^{n+1}}$-module.
    Finally, the inclusion ${\rm Sing}(M_L(u)) \subseteq V(E_L)$ follows from Lemma \ref{lem:ELisprY}.
\end{proof}

We conjecture that the inclusion ${\rm Char}(M_L(u)) \subseteq W$ is an equality (Conjecture~\ref{conj:bigconjecture}(2)). 

\begin{example} \label{ex:M2}
The following Macaulay2 \cite{M2} code snippet shows how we can confirm Theorem \ref{thm:holonomicity} for the matroid hypergeometric system $M_L(u)$ when $L$ is defined by \eqref{eq:generic matroid}. The generators of the matroid hypergeometric ideal are displayed below in \eqref{eq:operatorsbanana} (setting $n = 3$).
The command \texttt{holonomicRank} computes the holonomic rank of a $D$-module.
The output of this command shows that the matroid hypergeometric system $M_L(u)$ has holonomic rank $7$ in this case.
\begin{lstlisting}
D = QQ[z_0..z_3, dz_0..dz_3, WeylAlgebra => {z_0..z_3 => dz_0..dz_3}]
u = {1/2, 1/3, 1/5, 1/7}; s = -sum(u) -3
-- Set up the Weyl algebra and choose parameters
A = matrix{{-1, 1, 0, 0},{-1, 0, 1, 0},{-1, 0, 0, 1}}
H = sum(0..3, i -> z_i*dz_i) + s
Ps = apply(entries A, row -> ( sum(0..3, j -> row#j * z_j * dz_j^2) 
- sum(0..3, j -> row#j * u#j * dz_j) ))
Q = sum(subsets({dz_0, dz_1, dz_2, dz_3}, 3), product)
I = ideal(flatten {H, Ps, Q}) -- The matroid hypergeometric ideal
print holonomicRank I -- The holonomic rank is 7.
J = DsingularLocus I -- Computes the defining ideal of the singular locus
factor (gens J)_0_0
\end{lstlisting}
The last output is $z_0z_1z_2z_3\Delta(L^{-1})$, which matches Theorem \ref{thm:holonomicity} by Example \ref{ex:introcontd}.
\end{example}

A \emph{solution} to the matroid hypergeometric system $M_L(u)$ is a holomorphic function $g(z)$ which is annihilated by the operators \eqref{eq:op1}-\eqref{eq:op3}. In the second part of this section, we identify
\emph{Euler integrals} \cite{Matsubara-Heo2023Four} as solutions to $M_L(u)$.
We consider the following projective {Euler integral}:
\begin{equation} \label{eq:eulerint} g_\Gamma(z) \, = \, \int_\Gamma f(t, z)^{-s} \, \prod_{i = 0}^n \ell_i(t)^{u_i} \, \Omega(t) \, = \, \int_{\Gamma} \varphi(t,z) \, f(t,z)^{d+1} \, \Omega(t).\end{equation}
We start by parsing the notation. 
We let $\Omega(t) = \sum_{i = 0}^d (-1)^i \, t_i\,{\rm d}t_0 \wedge \cdots \wedge \widehat{{\rm d}t_i} \wedge \cdots \wedge {\rm d} t_d$ be the standard form on $\mathbb{P}^d$. It is the unique global section of  $\Omega_{\mathbb{P}^d}^d(d+1)\simeq \mathcal{O}_{\mathbb{P}^d}$.
The notation $\widehat{\cdot}$ means that ${\rm d}t_i$ is omitted.
The exponents $s, u_i$ are as above, and $\ell_i \in \mathbb{C}[t_0, \ldots, t_d]_1$ is the linear form corresponding to the $i$-th column of a matrix $A \in \mathbb{C}^{(d+1) \times (n+1)}$. Their arrangement is denoted by ${\cal A} = \{ t \in \mathbb{P}^d \,:\, \ell_0(t) \cdots \ell_n(t) = 0 \}$. We use $f$ to denote the rational function $f(t,z) = \sum_{i = 0}^n z_i \,  \ell_i^{-1}(t)$. For the integrand in  \eqref{eq:eulerint} to descend to a form on $\mathbb{P}^d$, we must impose the homogeneity condition $\lambda^{-d-1}\varphi(\lambda t, z) \, \Omega(\lambda t) = \varphi(t,z) \, \Omega(t)$, which matches \eqref{eq:homogeneity}.
With this condition, the integrand $\varphi = f^{-s-d-1} \prod_{i=0}^n \ell_i^{u_i}$ is a multi-valued function of $t \in \mathbb{P}^d \setminus {\cal A}$ for each fixed $z$. The integration cycle $\Gamma$ contains the information on which branch to integrate. More precisely, $\Gamma$ is a \emph{twisted cyle} from the theory of twisted (co-)homology on the space $X_z = \{ t \in \mathbb{P}^d \setminus {\cal A} \, : \, f(t,z)  \neq 0 \}$ \cite[Section 3]{Matsubara-Heo2023Four}. To match our notation with \cite{Matsubara-Heo2023Four}, one may dehomogenize by changing coordinates so that $\ell_i(t) = t_i$ for $i = 0, \ldots, d$ and setting $t_0 = 1$.

We now construct a $D$-module from \eqref{eq:eulerint}, following the well-known construction in \cite[Proposition 1.5.28 (i)]{hotta2007d}.
Define $X = \{ (t, z) \in (\mathbb{P}^d\setminus\mathcal{A})\times \mathbb{C}^{n+1} \, : \, f(t,z) \neq 0\}$ and let $\Omegarel{p}$ be the $\mathcal{O}_{X}$-module of relative $p$-forms.
Concretely, an element $\eta\in \Omegarel{p}$ is of the form
\begin{equation}\label{eq:relative p-form}
\eta=\sum_{\substack{I\subset\{ 0,\dots,d\}\\ |I|=p}}g_{I}(t,z)\bigwedge_{i\in I} {\rm d}t_{i},\ \quad g_{I}(t,z) \, \in \, \C\Big[ t_0,\dots,t_d,z_0,\dots,z_n,\frac{1}{f\ell_0\cdots\ell_n} \Big],
\end{equation}
where $g_{I}(t,z)$ is homogeneous of degree $(-p)$ in $t$, i.e., $g_I(\lambda t,z) = \lambda^{-p}g_I(t,z)$. 
We define
\[
\omega \, := \, d_t \log \varphi(t,z) \, = \,  -(s+d+1) \, \frac{d_tf(t,z)}{f(t,z)}+\sum_{j=0}^nu_j \, \frac{d_t\ell_j(t)}{\ell_j(t)} \, \in  \, \Omegarel{1}.
\]
Here, $d_t$ is the exterior derivative with respect to $t$.
The \emph{twisted differential} \[  \nabla_\omega: \, \Omegarel{p}\, \longrightarrow \,  \Omegarel{p+1}, \quad \eta \, \longmapsto \, d_t \, \eta+\omega\wedge \eta\, \] 
satisfies $\nabla_\omega^2=0$.
 The $D_{\C^{n+1}}$-module
$M_L^{\rm hyp}(u)$ is the $d$-th cohomology group $H^d(\Omegarel{\bullet},\nabla_\omega)$ of the \emph{twisted de Rham complex} $(\Omegarel{\bullet},\nabla_\omega )$.
The $D_{\C^{n+1}}$-module structure on $M_L^{\rm hyp}(u)$ is defined to mimic the action of the partial derivatives $\partial_i$ on Euler integrals \eqref{eq:eulerint}. That is, the identity $\partial_i \int_{\Gamma} \varphi \eta \, = \, \int_\Gamma\varphi(\partial_i\eta-(s+d+1)\frac{\partial_if}{f}\eta)$ motivates the following definition
\[
\partial_i\bullet\left[\eta\right]:=\left[\partial_i\eta-(s+d+1)\frac{\partial_if(t,z)}{f(t,z)}\eta\right],\quad [\eta]\in M_L^{\rm hyp}(u),\ i=0,\dots,n.
\]
Here $\partial_i\eta\in \Omega^d_{X/Z}$ is  obtained by applying $\partial_i$ to all the coefficients $g_I(t,z)$ in \eqref{eq:relative p-form}.
The next proposition relates $M^{\rm hyp}_L(u)$ with the matroid hypergeometric system $M_L(u)$.

\begin{proposition}\label{prop:psi}
    There is a unique $D_Z$-module homomorphism $\psi:M_L(u)\to M^{\rm hyp}_L(u)$ which sends the class $[1]\in M_L(u)$ to $\left[ f(t,z)^{d+1}\Omega(t)\right]\in M^{\rm hyp}_L(u)$.
    In particular, each differential operator in the matroid hypergeometric ideal $H_L(u)$ annihilates the Euler integral \eqref{eq:Euler integral}.
\end{proposition}
\begin{proof}
    It suffices to prove that $\psi: [1] \mapsto [f(t,z)^{d+1}\Omega(t)]$ is well-defined. 
    It is clear that $H \bullet \left[ f(t,z)^{d+1}\Omega(t)\right] = 0$ since $f(t,z)$ is homogeneous of degree $1$ in $z$. 
    Next, we show that $P_i \bullet \left[ f(t,z)^{d+1}\Omega(t)\right] = 0$.
    For any $i = 0, \ldots, d$, we have
    \[ {P}_i \bullet \left[ f(t,z)^{d+1}\Omega(t)\right] \, = \, -s \left[\Big( -(s+1) \frac{\partial f}{\partial t_i} f^{-2} + \sum_{j = 0}^n u_j \frac{\partial \ell_j}{\partial t_i} \ell_j^{-1} f^{-1} \Big )f^{d+1}\Omega(t)\right]. \]
    We must show that the following class is zero in $M^{\rm hyp}_L(u)$:
    \[ \xi \, = \, \left[\Big( -(s+1) \frac{\partial f}{\partial t_i} f^{-2} + \sum_{j = 0}^n u_j \frac{\partial \ell_j}{\partial t_i} \ell_j^{-1} f^{-1} \Big )f^{d+1} \, \Omega(t)\right].\]
    Namely, we must show that $\xi$ can be written as $\nabla_\omega \xi'$ for some $\xi'\in\Omegarel{d-1}$. 
    We calculate that 
    \[ \xi' \, = \, f^d\sum_{j\neq i} {\rm sgn}(i,j)t_j{\rm d}t_1 \wedge \cdots \wedge \widehat{{\rm d} t_i}\wedge \cdots \wedge \widehat{{\rm d} t_j} \wedge \cdots \wedge {\rm d}t_d,\quad 
    {\rm sgn}(i,j):=
    \begin{cases}
        (-1)^{i+j}&(j<i)\\
        (-1)^{i+j+1}&(i<j).
    \end{cases}
    \]
    This shows $P_i \bullet \left[ f(t,z)^{d+1}\Omega(t)\right] = 0$. Finally, we must prove that $Q_h \bullet \left[ f(t,z)^{d+1}\Omega(t)\right] = 0$. Let $d_h$ be the degree of a homogeneous polynomial $h \in I(L^{-1})$. 
    The claim follows from
    \[ Q_h \bullet \left[ f(t,z)^{d+1}\Omega(t)\right] \, = \, (-s)(-s-1) \cdots (-s-d_h+1)\left[ h\Big(\frac{1}{\ell_0}, \ldots, \frac{1}{\ell_n} \Big) f(t,z)^{d-d_h+1}\Omega(t)\right]\]
    and $h(\ell_0^{-1}, \ldots, \ell_n^{-1}) = 0$ by the definition of $L^{-1}$. 
\end{proof}
\noindent
We expect that the $D_{\mathbb{C}^{n+1}}$-module homomorphism $\psi:M_L(u)\to M^{\rm hyp}_L(u)$ from Proposition \ref{prop:psi} is in fact an isomorphism.
This is Conjecture \ref{conj:Integral representation}.

Let us show that certain Feynman integrals from particle physics are of the form \eqref{eq:eulerint}.

\begin{example}\label{ex:banana integral}
The Feynman integral of a \emph{banana graph} with $n$ internal edges (in Feynman-parameter representation with dimensional and analytic regularization \cite[Section 2.5.3]{weinzierl2022feynman})~is 
\[ {\cal I}_n(z) \, = \, \int_{\mathbb{R}^{n-1}_+} \frac{{\cal U}_n(\alpha)^\mu}{{\cal F}_n(\alpha,z)^\nu} \, \Omega(\alpha),  \]
where $\alpha = (\alpha_1, \ldots, \alpha_n)$ and ${\cal U}_n, {\cal F}_n$ are the first and second \emph{Symanzik polynomials}:
\begin{equation}
{\cal U}_n(\alpha) \, = \, \sum_{i = 1}^{n} \prod_{\substack{j \in \{1, \ldots, n\} \\ j \neq i}} \alpha_j, \quad {\cal F}_n(\alpha,z) \, = \, (z_1 \, \alpha_1 + \cdots + z_n \, \alpha_n) \, {\cal U}_n(\alpha) \, - \, z_0 \, \alpha_1 \cdots \alpha_n.
\end{equation}
The integration contour is the positive orthant $\mathbb{R}^{n-1}_+$ in $\mathbb{P}^{n-1}$, and the exponents $\mu, \nu \in \mathbb{C}$ satisfy the homogeneity condition $\mu (n-1) - \nu n + n = 0$. 
In physics, the parameters $z_1, \ldots, z_n$ are the squared masses of the \emph{internal particles} travelling along edges $1, \ldots, n$, and $z_0 = p^2$ is the squared Minkowski norm of the momentum vector $p$. We rewrite ${\cal I}_n(z)$ as
\begin{align*} 
{\cal I}_n(z) &\, = \, \int_{\mathbb{R}^{n-1}_+} \Big ( \alpha_1 \cdots \alpha_n \big (\tfrac{1}{\alpha_1} + \cdots  + \tfrac{1}{\alpha_n} \big )\Big )^{\mu - \nu} \Big ( z_1  \alpha_1 + \cdots + z_n \alpha_n - \frac{z_0}{\big( \tfrac{1}{\alpha_1} + \cdots +\tfrac{1}{\alpha_n}  \big) }\Big )^{-\nu} \, \Omega(\alpha).
\end{align*}
Changing coordinates by setting $t_{i-1} = \alpha_i^{-1}$ and $d = n-1$ we obtain 
    \begin{equation}\label{eq:banana integral}
      {\cal I}_n(z)  = (-1)^{u_0} \int_{\Gamma} \Big ( \frac{z_0}{-t_0 - \cdots - t_d} + \frac{z_1}{t_0} + \cdots + \frac{z_n}{t_d}\Big )^{-s}  (-t_0-\cdots-t_d)^{u_{0}} t_0^{u_1}\cdots t_d^{u_n}  \Omega(t),
    \end{equation}
where $s=\nu,$ $u_{0}=\mu-\nu$, $u_1=\cdots=u_n=\nu-\mu-2$, and $\Gamma=\mathbb{R}^{n-1}_+$.
Up to a constant, this is the integral \eqref{eq:eulerint} for the $(d+1) \times (n+1)$ matrix $A = \begin{pmatrix} -\sum_{i = 1}^n e_i & e_1 & \cdots & e_n \end{pmatrix}$. Proposition \ref{prop:psi} gives the following $n+2$ annihilators of \eqref{eq:banana integral}: 
    \begin{align}
    \label{eq:operatorsbanana}
       \begin{split} H &\, = \, z_0 \, \partial_0 + \cdots + z_n \, \partial_n + s, \quad Q \, = \, \sum_{i=0}^n\partial_0\cdots\widehat{\partial_i}\cdots\partial_n,\\ 
        P_i & \, = \, -z_0\partial_0^2+z_{i}\partial_{i}^2- (-u_0 \partial_0 +u_{i} \partial_{i}), \quad i = 1, \dots, n. 
        \end{split}
    \end{align} 
More generally, the operators $H, Q, P_i$ of the function ${\cal I}_n(z)$ were recently found in the physics literature, see \cite[Equations (3.4) and (3.5)]{Flieger}.
\end{example}
Example \ref{ex:Lauricella} deals with the banana integral for $n = 3$ and shows that it recovers the Lauricella function $F_C^{(3)}$. We end the section by explaining how this extends to general $n$. 

\begin{example}
In this example, we assume that $\Gamma$ is such that the following identity holds for any $(u_0,\dots,u_{n}) \in (\C \setminus \mathbb{Z})^{n+1},(m_1,\dots,m_n) \in \mathbb{Z}^{n}$, and $s=-\sum_{i=0}^{n}u_i-n$:
\begin{equation}\label{eq:Dirichlet integral}
\int_{\Gamma}(t_0+\cdots+t_d)^{s+u_{0}+\sum_{i=1}^n m_i}t_0^{u_1-m_1}\cdots t_d^{u_n-m_n} \, \Omega(t)
\, = \, 
\frac{(1+s+u_{0})_{\sum_{i=1}^n m_i}}{(-u_1)_{m_1}\cdots(-u_n)_{m_n}}.    
\end{equation}
Here $(a)_m:=a\cdot (a+1)\cdots(a+m-1)$
is the Pochhammer symbol and $d = n-1$. We denote the function \eqref{eq:Dirichlet integral} by $I(u;m)$. A cycle $\Gamma$ satisfying our assumption can be constructed as in \cite[Section 7]{beukers2010algebraic}.
When the ratios $|\frac{z_1}{z_{0}}|,\dots,|\frac{z_n}{z_{0}}|$ are sufficiently small, the contour $\Gamma$ can be regarded as a twisted cycle of the integral \eqref{eq:eulerint}, which we denote by the same symbol $\Gamma$ by abuse of notation. 
Writing $t_+ = t_0 + \cdots + t_d$, it follows from the multinomial theorem that
\[
\begin{split}
    \Big ( \frac{-z_0}{t_+} + \frac{z_1}{t_0} + \cdots + \frac{z_n}{t_d}\Big )^{-s}
&=(-z_0)^{-s}t_+^{s}\Big ( 1- t_+\Big (\frac{z_1}{z_0t_0} + \cdots + \frac{z_n}{z_0t_d}\Big )\Big)^{-s}\\
&=(-z_0)^{-s}t_+^{s}
\sum_{m_1,\dots,m_n=0}^\infty \frac{(s)_{m_1+\cdots+m_n}}{m_1!\cdots m_n!}\, t_+^{\sum_{i=1}^nm_i}\, \left(\frac{z_1}{z_0t_0}\right)^{m_1}\cdots \left(\frac{z_n}{z_0t_d}\right)^{m_n}.
\end{split}
\]
From this and \eqref{eq:Dirichlet integral}, we see that the integral $\mathcal{I}_n(z)$ in \eqref{eq:banana integral} has the following expansion: 
\begin{align*}
    \mathcal{I}_n(z)
    \, =& \, \,
    (-z_{0})^{-s}\sum_{m_1,\dots,m_n =0}^\infty \frac{(s)_{m_1+\cdots+m_n}}{m_1!\cdots m_n!}\left(\frac{z_1}{z_{0}}\right)^{m_1}\cdots\left(\frac{z_n}{z_{0}}\right)^{m_n}I(u;m)\\
    =& \, \, (-z_{0})^{-s}F_C^{(n)}\left(\substack{s,1+s+u_{0}\\ -u_1,\dots,-u_n};\frac{z_1}{z_{0}},\dots,\frac{z_n}{z_{0}}\right).
\end{align*}
Here we used the Lauricella function $F_C^{(n)}$ displayed in Example \ref{ex:Lauricella}.
\end{example}

\section{Euler discriminants and characteristic varieties}\label{sec:6}
In this section, we relate the vanishing locus of the principal matroid determinant to the {\it Euler discriminant locus} of the family of hypersurfaces in \eqref{eq:family of hypersurfaces}. This terminology first appeared in \cite[Definition 3.1]{esterov2013discriminant}. While the main result (Theorem \ref{thm:ED}) is purely geometric, our proof uses $D$-module theory. In particular, we use recent results from \cite{matsubara2025hypergeometric} relating the Euler discriminant to the singular locus of the $D$-module $M_L^{\rm hyp}(u) = H^d(\Omega^\bullet_{X/Z},\nabla_\omega)$ 
introduced in Section~\ref{sec:5}.

The base of the family \eqref{eq:family of hypersurfaces} is the affine space of parameters $Z =\C^{n+1}={\rm Spec}\,\C[z_0,\dots,z_n]$ and $Z^\vee =\C^{n+1}={\rm Spec}\,\C[x_0,\dots,x_n]$ is the dual vector space of $Z$. 
Matching the notation in previous sections, we have $\mathbb{P}^n=(Z^\vee\setminus\{0\})/\C^\times$ and $(\mathbb{P}^n)^\vee=(Z\setminus\{0\})/\C^\times$. Throughout the section, $L \subset \mathbb{P}^n$ is a linear space satisfying \eqref{eq:assum_noloops}.  Its reciprocal variety is $L^{-1} \subset \mathbb{P}^n$.

For each $z \in Z$, we define $H_z \, = \, \{ x \in \mathbb{P}^n \, : \, z_0x_0 + \cdots + z_n x_n = 0 \}$ 
 and we write $\chi_z$ for the signed Euler characteristic $(-1)^{d-1}\chi(V_z)$ of the very affine variety $V_z\simeq L^{-1}\cap T\cap H_z$ defined in \eqref{eq:family of hypersurfaces}. 
 Note that $\chi_z\geq0$ by  \cite[Theorem 1(iii)]{huh2013maximum}.
Let $\chi^*$ denote the maximum value of the function $z \mapsto \chi_z$.
The Euler discriminant locus $\nabla_\chi(Z)$ is defined by the following formula:
\begin{equation} \label{eq:eulerdisc}
\nabla_\chi:=\{z\in Z \, : \,  \chi_z<\chi^*\}.
\end{equation}
It follows from \cite[Theorem 3.1]{telen2024euler} that $\nabla_\chi$ is a closed subvariety of $Z$.

\begin{theorem}\label{thm:ED}
    The Euler discriminant of the family \eqref{eq:family of hypersurfaces} is the vanishing locus of the principal matroid determinant.
    In symbols, the identity $\nabla_\chi = \{z \in Z \, : \, E_L(z) = 0 \}$ holds true.
\end{theorem}

In what follows, we shall utilize the functorial realization of the $D_Z$-module $M^{\rm hyp}_L(u)$ introduced in Section \ref{sec:5}.
We use standard notation for various functors on $D$-modules as in \cite{hotta2007d}.
In particular, for a smooth affine variety $Y$, $D_Y$ is the ring of differential operators on $Y$.
Moreover, given a multivalued function $\varphi$ on $Y$ such that $\frac{d\varphi}{\varphi}\in\Omega^1_Y$, we write $\mathcal{O}_Y\varphi$ for the $D_Y$-module which is isomorphic to $\mathcal{O}_Y$ as an $\mathcal{O}_Y$-module, and $D_Y$ acts as follows:
\[
\partial\left( f\varphi\right)=\left(\partial f+\frac{\partial\varphi}{\varphi}\right)\varphi,\quad\quad f\in\mathcal{O}_Y,
\]
where $\partial$ is any global vector field on $Y$.
Let $X$ be the variety $((\mathbb{P}^d\setminus\mathcal{A})\times Z)\setminus  V_{(\mathbb{P}^d\setminus\mathcal{A})\times Z}(f)$ and let $\pi:X\to Z$ be the coordinate projection. 
Here $f(t,z) = \sum_{i = 0}^n z_i \, \ell_i(t)^{-1}$ is viewed as a regular function on the product space $(\mathbb{P}^d\setminus\mathcal{A})\times Z$. 
By \cite[Proposition 1.5.28 (i)]{hotta2007d}, there is an isomorphism of $D_Z$-modules 
$$M_L^{\rm hyp}(u)\, \simeq \, H^0\int_{\pi}\mathcal{O}_X\, f(t,z)^{-s-d-1}\prod_{j=0}^n\ell_j(t)^{u_j} \, = \, H^0 \int_\pi {\cal O}_X \, \varphi(t,z),$$
where $s$ and $u$ are related as in \eqref{eq:homogeneity}, and $\varphi(t,z) = f^{-s-d-1} \prod_{i=0}^n\ell_i^{u_i}$ is the multivalued function from Equation \eqref{eq:eulerint}. 
This is the \emph{hypergeometric $D$-module} from \cite{matsubara2025hypergeometric}.

The first step in proving Theorem \ref{thm:ED} is to identify $\nabla_\chi$ as the singular locus of $M^{\rm hyp}_L(u)$.
The statement of the following Proposition uses the \emph{beta invariant} $\beta(L)$ of the matroid $M(L)$.

\begin{proposition}\label{prop:local cohomology}
    For generic $u\in \C^{n+1}$, the holonomic rank of $M^{\rm hyp}_L(u)$ is $\tilde{\chi}^*:=\chi^*+\beta(L)$ and the singular locus of $M^{\rm hyp}_L(u)$ is the Euler discriminant locus $\nabla_\chi$ from \eqref{eq:eulerdisc}.
\end{proposition}
\begin{proof}
We define a closed subvariety $X'\subset (\C^\times)\times T\times Z$ of pure codimension $n-d+1$:
$$
X' \, := \, \{ (\tau,x,z)\in (\C^\times)\times T\times Z \, : \,  \tau=\langle x,z\rangle, \, x\in L^{-1}\} \, \simeq \, X.
$$
The {\it local cohomology} $\mathcal{B}_{X'|(\C^\times)\times T\times Z}$ is $H^0\int_{\iota'}\mathcal{O}_{X'}$ where $\iota':X'\to (\C^\times)\times T\times Z$ is the natural embedding.
Let $\pi':(\C^\times)\times T\times Z\to Z$ denote the coordinate projection. 
A discussion analogous to \cite[Proposition 3.1]{matsubara2025hypergeometric} shows that there is an isomorphism 
$$\int_{\pi}\mathcal{O}_X \, \varphi(t,z) \simeq \int_{\pi'}(\mathcal{B}_{Y|(\C^\times)\times T\times Z})\otimes\mathcal{O}_{(\C^\times)\times T\times Z} \, \tau^{-s-d-1}\prod_{j=0}^nx_j^{-u_j}.$$
This also proves that these complexes of $D_Z$-modules are concentrated in degree $0$ by \cite[Lemma 2.4]{matsubara2025hypergeometric}.
The Euler characteristic $\tilde{\chi}_z:=(-1)^d\chi(L^{-1}\cap T\setminus H_z)$ is related to $\chi_z$ as follows:
\begin{equation}\label{eq:chi and chi tilde}
   \hspace{-.7em} \chi_z=(-1)^{d-1}\chi(L^{-1}\cap T\cap H_z)=(-1)^{d-1}(\chi(L^{-1}\cap T)-\chi(L^{-1}\cap T\setminus H_z))=\tilde{\chi}_z-\beta(L).
\end{equation}
The statement on the holonomic rank follows from the displayed formula before \cite[Lemma 2.10]{matsubara2025hypergeometric}.
Finally, the equality ${\rm Sing}(M^{\rm hyp}_L(u))=\nabla_\chi$ follows from \cite[Corollary 2.13]{matsubara2025hypergeometric}.
\end{proof}

Proposition \ref{prop:local cohomology} reduces the proof of Theorem \ref{thm:ED} to showing the identity ${\rm Sing}(M^{\rm hyp}_L(u)) = \{ z \in Z \, : \, E_L(z) = 0 \}$.
For this, our strategy is to use \cite[Theorem 3.4]{FraneckiKapranov}. To apply that theorem, we must first identify $M^{\rm hyp}_L(u)$ with another $D_Z$-module, obtained via a functorial realization of the Fourier transform \cite[Lemma 7.1.4]{katz1985transformation}. 
We start from the canonical identification 
\begin{equation}\label{eq:identification}
    T^*Z\, = \, Z^\vee\times Z \, = \, T^*Z^\vee.
\end{equation}
Let $p_1:Z^\vee\times Z\to Z^\vee$ and $p_2:Z^\vee\times Z\to Z$ be the canonical projections, and let $\langle x,z\rangle$ denote the duality pairing of $x\in Z^\vee$ and $z\in Z$.
Then, the inverse Fourier transform is defined by
\[
\FL^{-1}(M)\, := \,  \int_{p_2}p_1^\dagger M\otimes \mathcal{O}_{Z^\vee\times Z}e^{-\langle x,z\rangle}
\]
for any holonomic $D_{Z^\vee}$-module $M$.
Below, we write $\ell^{-1}:{\rm Cone}(\P^d\setminus\mathcal{A})\to{\rm Cone}(L^{-1}\cap T)$ for the map defined by $t=(t_0,\dots,t_d)\mapsto (\ell_0(t)^{-1},\dots,\ell_n(t)^{-1})$ and $j:{\rm Cone}(L^{-1}\cap T)\to(\C^\times)^{n+1}$ is the natural inclusion.
Furthermore, let $\iota:(\C^\times)^{n+1}\to Z^\vee$ be the natural inclusion.

\begin{theorem}\label{thm:integral rep2}
    Assume that $s\notin \mathbb{Z}$.
    Then, there is a canonical isomorphism     \begin{equation}\label{eq:Euler integral D-module}
        \FL^{-1}\left(\int_{\iota}\mathscr{M}(u)\right)   \simeq M^{\rm hyp}_L(u),\quad \text{where}\quad \mathscr{M}(u):=\left( \int_{j\circ\ell^{-1}}\mathcal{O}_{{\rm Cone}(\P^d\setminus\mathcal{A})}\right)\otimes\mathcal{O}_{(\C^\times)^{n+1}}\prod_{j=0}^nx_j^{-u_j}.
    \end{equation}
\end{theorem}

\begin{proof}

To prove the isomorphism \eqref{eq:Euler integral D-module}, we consider a Cartesian diagram
    $$
    \xymatrix{
        {\rm Cone}(\P^d\setminus\mathcal{A})\times Z\ar[r]^-{\tilde{j}}\ar[d]_-{\tilde{p}}&Z^\vee\times Z\ar[d]^{p_1}\\
        {\rm Cone}(\P^d\setminus\mathcal{A})\ar[r]^-{\iota\circ j\circ\ell^{-1}}&Z^\vee,
    }
    $$
    where $\tilde{p}$ is the projection and $\tilde{j}(t,z):=(\ell^{-1}(t),z)$.
    We choose homogeneous coordinates $t_0,\dots,t_d$ so that $\ell_0(t)=t_0$.
    Let $g:{\rm Cone}(\P^d\setminus\mathcal{A})\times Z\to {\rm Cone}(\P^d\setminus\mathcal{A})\times Z$ be given by $g(t_0,\dots,t_d,z)=(t_0,t_0t_1,\dots,t_0t_d,z)$.
    Using the relation $p_2\circ\tilde{j}\circ g=p_2\circ\tilde{j}$, the composition property of the direct image functor, the projection formula, and the base change formula \cite[Proposition 1.5.21, Corollary 1.7.5 and Theorem 1.7.3]{hotta2007d}, we obtain the following isomorphisms:
    \begin{fleqn}[-5pt]
        \begin{eqnarray*}
        \vspace{.7em}
            &&\ \ \ \ {\FL}^{-1}\left( \int_{\iota}\mathscr{M}(u)\right)\nonumber\\
            \vspace{.7em}
            &(\because\quad \text{composition property})&\simeq\,\,{\FL}^{-1}\left( \int_{\iota\circ j\circ\ell^{-1}}\mathcal{O}_{{\rm Cone}(\P^d\setminus\mathcal{A})}\displaystyle\prod_{j=0}^n\ell_j(t)^{u_j}\right)\nonumber\\
            \vspace{.7em}
            &(\because\quad \text{projection  formula})&\simeq\int_{p_2}\mathcal{O}_{Z^\vee\times Z}e^{-\langle x,z\rangle}\otimes p_1^\dagger\left[ \int_{\iota\circ j\circ\ell^{-1}}\mathcal{O}_{{\rm Cone}(\P^d\setminus\mathcal{A})}\displaystyle\prod_{j=0}^n\ell_j(t)^{u_j}\right] \nonumber\\
            \vspace{.7em}
            &(\because\quad \text{base change formula})\quad&\simeq \int_{p_2}\mathcal{O}_{Z^\vee\times Z}e^{-\langle x,z\rangle}\otimes \int_{\tilde{j}}\tilde{p}^\dagger\left[ \mathcal{O}_{{\rm Cone}(\P^d\setminus\mathcal{A})}\displaystyle\prod_{j=0}^n\ell_j(t)^{u_j}\right]\nonumber\\
            \vspace{.7em}
            &(\because\quad \text{projection formula})\quad&\simeq\int_{p_2}\circ \int_{\tilde{j}}\mathcal{O}_{{\rm Cone}(\P^d\setminus\mathcal{A})\times Z}e^{-\sum_{j=0}^n\frac{z_j}{\ell_j(t)}}\displaystyle\prod_{j=0}^n\ell_j(t)^{u_j}\nonumber\\
            \vspace{.7em}
            &\left(\because\quad \text{composition property}\right)&\simeq\int_{p_2\circ \tilde{j}}g^\dagger\left[\mathcal{O}_{{\rm Cone}(\P^d\setminus\mathcal{A})\times Z}e^{-\sum_{j=0}^n\frac{z_j}{\ell_j(t)}}\displaystyle\prod_{j=0}^n\ell_j(t)^{u_j}\right]\nonumber\\
            \vspace{.7em}&&\simeq\int_{p_2\circ\tilde{j}}\mathcal{O}_{{\rm Cone}(\P^d\setminus\mathcal{A})\times Z}e^{-t_0^{-1}\left(z_0+\sum_{j=1}^n\frac{z_j}{\ell_j(1,t')}\right)}t_0^{-s-d-1}\displaystyle\prod_{j=0}^n\ell_j(1,t')^{u_j}\nonumber\\
            \vspace{.7em}
            &&\simeq\int_{\pi}\mathcal{O}_X f(t;z)^{-s-d-1}\prod_{j=0}^n\ell_j(t)^{u_j}.
            \end{eqnarray*}
\end{fleqn}
    In the last step, we used a standard transformation of $D$-modules (see, e.g., \cite[Proposition 2.1]{matsubara2020euler}).
    This proves the isomorphism \eqref{eq:Euler integral D-module}.
\end{proof}

We say that a subset $C\subset Z^\vee\times Z$ is \emph{biconic} if it is stable under the action of $\C^\times\times\C^\times$ on $Z^\vee\times Z$ given by the component-wise scaling.

\begin{corollary}\label{cor:ED=sing}
    For generic choices of $u\in\C^{n+1}$, 
    the characteristic variety ${\rm Char}(\int_\iota\mathscr{M}(u))$ is biconic and identified with ${\rm Char}(M^{\rm hyp}_L(u))$ via \eqref{eq:identification}.
    Moreover, the identity $p_2(\overline{{\rm Char}(\int_\iota\mathscr{M}(u))\setminus T_Z^*Z})=\nabla_\chi$ holds true, where the overline is the closure in $T^*Z=Z^\vee\times Z$. 
\end{corollary}
\begin{proof}
The functor $\int_{\pi}$ preserves regularity \cite[Theorem 6.15]{hotta2007d}. Hence, since ${\cal O}_X \varphi$ is a regular holonomic $D_Z$-module, $M^{\rm hyp}_L(u)$ regular holonomic as well.
It follows from Theorem \ref{thm:integral rep2} and \cite[Theorem 7.24]{brylinski1982transformations} that $\mathscr{M}(u)$ is a monodromic $D_Z$-module \cite[Proposition 7.12]{brylinski1982transformations}.
Thus, Theorem \ref{thm:integral rep2} and \cite[Theorem 7.25]{brylinski1982transformations} show that the characteristic variety of $\FL^{-1}(\int_\iota\mathscr{M}(u))\simeq M^{\rm hyp}_L(u)$ is identified with ${\rm Char}\left(\int_\iota\mathscr{M}(u)\right)$ via \eqref{eq:identification}.
This also proves that ${\rm Char}(\int_\iota\mathscr{M}(u))$ is biconic.
The equality $p_2(\overline{{\rm Char}(\int_\iota\mathscr{M}(u))\setminus T_Z^*Z})=\nabla_\chi$ follows from the identity $p_2(\overline{{\rm Char}(\int_\iota\mathscr{M}(u))\setminus T_Z^*Z})={\rm Sing}(\FL^{-1}\int_\iota\mathscr{M}(u))$, Proposition \ref{prop:local cohomology} and Theorem \ref{thm:integral rep2}.
\end{proof}

To compute the characteristic variety ${\rm Char}(\int_\iota\mathscr{M}(u))$ we use \cite[Theorem 3.4]{FraneckiKapranov} (see also \cite[Theorem 3.6]{maxim2024logarithmic}).
Note that \cite[Theorem 3.4]{FraneckiKapranov} is stated in terms of characteristic varieties of constructible sheaves while our exposition uses $D$-modules.
These two viewpoints are compatible via the Riemann-Hilbert correspondence as in \cite[Theorem 11.3.3]{kashiwara2013sheaves}.

\begin{theorem}\label{thm:characteristic variety}
For generic choices of $u \in \mathbb{C}^{n+1}$, the following equality holds:    \begin{equation}\label{eq:characteristic variety}
        {\rm Char}\left(\int_{\iota}\mathscr{M}(u)\right)=\bigsqcup_{F\in\mathcal{F}(M(L))}{\rm Con}(L^{-1}\cap T_F)=\bigcup_{F\in\mathcal{F}(M(L))}{\rm Con}(L^{-1}\cap \Lambda_F).
    \end{equation}
\end{theorem}

\begin{proof}
The second equality follows from the definition of the conormal bundle.
We prove the first equality.
Since $\mathscr{M}(u)$ is regular holonomic and ${\rm DR}_{(Z^{\vee})^{\rm an}}(\int_\iota\mathscr{M}(u))=\mathbb{R}\iota_*({\rm DR}_{(\C^\times)^{n+1}}(\mathscr{M}(u)))$, where ${\rm DR}$ stands for the analytic de Rham functor, we can apply Theorem 3.4 of \cite{FraneckiKapranov} 
    to compute 
    ${\rm Char}(\int_{\iota}\mathscr{M}(u))$. For transparency, we note that in the notation $X, U, f$ from \cite[Section 3]{FraneckiKapranov}, we apply \cite[Theorem 3.4]{FraneckiKapranov}  for $X = Z^\vee, \, U = (\mathbb{C}^\times)^{n+1} \subset X, \, f = x_0x_1 \cdots x_n$.
We define the variety
\[ \Lambda^\# \, = \, \Big \{ (x, z,s) \in U \times Z \times \mathbb{C}^{n+1} \, : \, z = \Big(z'_i + \frac{s_i}{x_i} \Big)_{i = 0, \ldots, n}, (x,z') \in {\rm Con}(L^{-1} \cap T) \Big \}.\]
By Proposition \ref{prop:reciprocalHornKapranov}, this is parametrized by 
\begin{equation} \label{eq:paramLambdasharp} (L^{-1} \cap T) \times L^\perp \times \mathbb{C}^{n+1}  \rightarrow \Lambda^\#, \quad (x,p,s) \mapsto (x, p \star x^{-2}  + s \star x^{-1}),\end{equation}
where $x^{-k} = \phi_{-k}(x)$ and $\star$ is the Hadamard product.
Our characteristic variety is given by 
\[ {\rm Char} \Big (\int_{\iota}\mathscr{M}(u) \Big ) \, = \, \lim_{s \rightarrow 0} \, \Lambda^\#_s \, = \, {\rm pr}_{Z^\vee \times Z} \, \Big(\overline{\Lambda^\# \setminus \{s = 0\}} \cap \{s = 0 \}\Big).\]
Here, the closure is in $Z^\vee \times Z \times \mathbb{C}^{n+1}$ and ${\rm pr}_{Z^\vee \times Z}$ is the coordinate projection.

Each point in $\lim_{s \rightarrow 0} \Lambda^\#_s$ obtained as the limit point of a curve segment in $\Lambda^\#$. That is, by \cite[Lemma 2.5]{sattelberger2023maximum}, a point $(x,z) \in \lim_{s \rightarrow 0} \Lambda^\#_s$ is given by 
\[ (x,z) \, = \, \lim_{\epsilon \rightarrow 0} \Big(x_i(\epsilon), \, z_i(\epsilon) \, = \, \frac{p_i(\epsilon)}{x_i(\epsilon)^2} + \frac{\epsilon s_i(\epsilon)}{x_i(\epsilon)} \Big)_{i = 0, \ldots, n}.\]
Here $x_i(\epsilon), p_i(\epsilon), s_i(\epsilon) \in \mathbb{C}[\![\epsilon]\!]$ are power series, and we used the parametrization \eqref{eq:paramLambdasharp}.
The symbol $\lim_{\epsilon\to0}$ should be interpreted in the sense of scheme theory: $x_i(\epsilon)$ is a morphism ${\rm Spec}\, (\C[\![\epsilon]\!])\to {\rm Spec}\, \C[x_i]$ and $\lim_{\epsilon\to 0}x_i(\epsilon)$ is the image of the maximal ideal of $\C[\![\epsilon]\!]$.
We have $p(\epsilon)= (p_i(\epsilon))_{i = 0, \ldots, n} \in L^\perp$.  It is useful to introduce the following additional notation: for a nonzero power series $g(\epsilon)$, let ${\rm ord}(g)$ be the order of $g$, i.e., the smallest exponent of $\epsilon$ appearing with a nonzero coefficient in $g$, and let $\tilde{g}(\epsilon)$ be defined as follows: $g(\epsilon)= \epsilon^{{\rm ord}(g)} \cdot \tilde{g}(\epsilon)$. In particular, for any $g(\epsilon) \in \mathbb{C} [\![\epsilon ]\!]$, we have $\lim_{\epsilon \rightarrow 0} \tilde{g}(\epsilon) \in \mathbb{C}^\times$.

By Proposition \ref{prop:stratify}, there exists a (possibly empty) flat $F \in {\cal F}(M(L))$ such that 
\[ (x_i(\epsilon), z_i(\epsilon)) \, = \,  \begin{cases}
    \Big (\tilde{x}_i(\epsilon), \frac{p_i(\epsilon)}{\tilde x_i(\epsilon)^2} + \frac{\epsilon  s_i(\epsilon)}{\tilde x_i(\epsilon)} \Big ) & i \in F \\
    \Big (\epsilon^{k_i}\tilde{x}_i(\epsilon), \frac{p_i(\epsilon)}{\epsilon^{2k_i} \tilde x_i(\epsilon)^2} + \epsilon^{1-k_i}\frac{s_i(\epsilon)}{\tilde x_i(\epsilon)} \Big ) & i \notin F 
\end{cases},\]
where $k_i = {\rm ord}(x_i(\epsilon))$.
For $i \in F$, we have $(x_i,z_i) = \lim_{\epsilon \rightarrow 0} (x_i(\epsilon),z_i(\epsilon)) = (x_i, p_ix_i^{-2})$. For $i \notin F$, notice that convergence of $z_i(\epsilon)$ in $Z$ for $\epsilon \to 0$ implies that 
\[ {\rm ord}(p_i(\epsilon)) \, \geq \,  \min \{ 0,1-k_i + {\rm ord}(s_i(\epsilon))\} +2k_i >0, \quad i \notin F. \] 
By continuity, $p_F = (p_i)_{i \in F}$ lies in $L_F^\perp$, so that $(x,z) \in {\rm Con}(L^{-1} \cap T_F)$. Moreover, choosing the curve $s(\epsilon)$ appropriately, it is clear that each point in ${\rm Con}(L^{-1} \cap T_F)$ arises in this way. 
\end{proof}

\begin{proof}[Proof of Theorem \ref{thm:ED}]
    We combine
Theorem \ref{thm:factorization}, Corollary \ref{cor:ED=sing} and Theorem \ref{thm:characteristic variety}.
\end{proof}

As an application of Theorem \ref{thm:characteristic variety}, we prove a formula for the {\it hypergeometric discriminant} associated with the Feynman banana integral from Example \ref{ex:banana integral} \cite{matsubara2025hypergeometric}. This is the Euler integral \eqref{eq:eulerint} corresponding to the linear space $L=\mathbb{P}{\rm Row}(A)$, with $A = \begin{pmatrix} -\sum_{i = 1}^n e_i & e_1 & \cdots & e_n \end{pmatrix}$.
By definition, the hypergeometric discriminant is the characteristic cycle of the $D_Z$-module $M^{\rm hyp}_L(u)$ for generic choices of $u$, which is indeed independent of $u$. We denote this cycle by $E^{\rm hyp}_{\texttt{B}_{n}}$ to match \cite[Section 5.1]{matsubara2025hypergeometric}, in which the following theorem is a conjecture.

\begin{theorem}
The hypergeometric discriminant of the integral ${\cal I}_n(z)$ in \eqref{eq:banana integral}, i.e., the characteristic cycle of $M_L^{\rm hyp}(u)$, is given by
\begin{equation}\label{eq:hypergeometric banana}
E^{\rm hyp}_{\texttt{B}_{n}}={\rm Con}(\nabla(L^{-1}))+m_0 \, {\rm Con}(Z)+\sum_{p=1}^{n-1}m_p \, \sum_{0\leq i_1<\cdots<i_p\leq n}{\rm Con}({V(z_{i_1},\dots,z_{i_p})}),    
\end{equation}
where $m_p=2^{n-p}-1$.
Moreover, $\nabla(L^{-1})$ is the Landau discriminant of the banana diagram $\texttt{B}_n$ as in \cite[Proposition 2]{mizera2022landau}.
\noindent
The variables $s,m_i$ in \cite[Proposition 2]{mizera2022landau} are $-z_0,\sqrt{-z_i}$ in \eqref{eq:hypergeometric banana}, and the variety $\nabla(L^{-1})$ is denoted by $V(\Delta_n)$ in the conjecture of \cite[Section 5.1]{matsubara2025hypergeometric}.
\end{theorem}

\begin{proof}
    The components appearing in \eqref{eq:hypergeometric banana} are exactly those in \eqref{eq:characteristic variety}.
    By \cite[Theorem 3.4]{FraneckiKapranov}, the multiplicity of the component ${\rm Con}(\nabla(L^{-1}))$ is the multiplicity of the corresponding component of the characteristic cycle of $\mathscr{M}(u)$.
    Since $j\circ\ell^{-1}:{\rm Cone}(\mathbb{P}^n\setminus\mathcal{A})\to (\C^\times)^{n+1}$ is a closed embedding, it follows that the multiplicity of the component ${\rm Con}(\nabla(L^{-1}))$ is one.

    To determine the other multiplicities, we use Kashiwara's index theorem \cite[Theorem~4.3.25]{dimca2004sheaves}, \cite{kashiwara1973index}.
    Let $m_{i_1\dots i_p}^p$ denote the multiplicity of the component $V(z_{i_1},\dots,z_{i_p})$ in $E^{\rm hyp}_{\texttt{B}_n}$.
    By the discussion in \cite[page 10]{matsubara2025hypergeometric},  $\tilde{\chi}_z = (-1)^d\chi(L^{-1}\cap T\setminus H_z)$ is given by the formula
    \begin{equation}\label{eq:Kashiwara}
    \tilde{\chi}_z=m_0-{\rm Eu}_{V(\Delta_n)}(z)+\sum_{p=1}^{n-1}(-1)^p\sum_{0\leq i_1<\cdots<i_p\leq n}m_{i_1\dots i_p}^p{\rm Eu}_{V(z_{i_1},\dots,z_{i_p})}(z),
    \end{equation}
    where ${\rm Eu}_Y$ denotes Euler obstruction of a subvariety $Y\subset Z$ \cite[Definition 4.1.36]{dimca2004sheaves}.
    Note that ${\rm Eu}_Y(y)=1$ for any smooth point $y\in Y$.
    By the computation in Example \ref{ex:banana integral} and \cite[Corollary 4.9]{BBKP}, we can use \cite[Proposition 55 and 56]{BBKP} to determine the multiplicities.
    Firstly, $m_0=2^{n}-1$ by \cite[Proposition 55]{BBKP}. 
    Let us fix $0\leq i_1<\cdots<i_p\leq n$ for some $1\leq p\leq n-1$ and choose a generic $z\in V(z_{i_1},\dots,z_{i_p})\setminus V(\Delta_n(z))$.
    Then, by \cite[Proposition 56]{BBKP} and \eqref{eq:Kashiwara},
    \begin{equation}\label{eq:recursion}
    2^{n-p}=\tilde{\chi}_z=2^n-1+\sum_{\emptyset\neq I\subset\{ i_1,\dots,i_p\}}(-1)^{|I|}m^{p}_{i_1\dots i_p}.    
    \end{equation}
    The relation \eqref{eq:recursion} recursively determines $m^p_{i_1\dots i_p}=m_p=2^{n-p}-1$.
\end{proof}

\begin{example}
    We continue Example \ref{ex:1.3}.
    Recall that $\chi^*=6$.
 The \emph{Euler stratification} introduced in \cite{telen2024euler} is a partition of $Z$ into a disjoint union of strata, so that the function $z \mapsto \chi_z$ is constant on each stratum. The strata are constructible \cite[Proposition 2.4]{telen2024euler}, and their closures are listed in the following table:
    \begin{center}
    \renewcommand{\arraystretch}{1.2} 
\begin{tabular}{c|l}
    ${\rm h}(\mathfrak{p})$ & \multicolumn{1}{c}{$\mathfrak{p}$}\\
    \hline
    1 & $\langle z_0\rangle_3,\langle z_1\rangle_3,\langle z_2\rangle_3,\langle z_3\rangle_3,\langle\Delta(L^{-1})\rangle_1$ \\
    \hline
    2 & $\begin{aligned}
        &\langle z_i,\, z_j\rangle_5\ \text{($0\leq i<j\leq 3$)},\quad\langle z_i,\delta(z_j,z_k,z_\ell)\rangle_4\ \text{($\{i,j,k,\ell\}=\{0,1,2,3\}$)},\;\\
        &\langle z_i - z_j,\, z_k - z_\ell\rangle_2\ \text{($\{i,j,k,\ell\}=\{0,1,2,3\}$)}
        \end{aligned}$ \\
        \hline
    3 & $\begin{aligned}
        &\langle z_i - z_j,\, z_k,\, z_\ell\rangle_6\ \text{($\{i,j,k,\ell\}=\{0,1,2,3\}$)},\quad\langle z_i,z_j,z_k\rangle_6\ \text{($0\leq i<j<k\leq 3$)}
        \end{aligned}$ 
\end{tabular}
\end{center}
Each stratum is encoded by a prime ideal $\mathfrak{p}$ of height $h(\mathfrak{p})$. The notation $\mathfrak{p}_k$ means that $\chi_z$ for a generic point $z \in V(\mathfrak{p})$ is $6-k$.  
The table also uses $\delta(x,y,z):=x^2+y^2+z^2-2xy-2xz-2yz$.
This stratification was computed using the Julia code provided in \cite{telen2024euler}. For instance, when $z$ is a generic point on the Steiner surface ${\cal S} = \nabla(L^{-1})$ from \eqref{eq:matroid discriminant for a uniform matroid}, the elliptic curve $\mathcal{E}$ has a nodal singularity and the Euler characteristic is $\chi_z=5$. The other strata are described in terms of the singular locus of the Steiner surface ${\cal S}$ as follows. For each partition $\{ i,j\}\cup\{ k,\ell\}=\{ 0,1,2,3\}$, ${\cal S}$ is singular along $V(z_i-z_j,z_k-z_\ell)$.
    There are exactly three such singular lines.
    Each of these contains two pinching points $z^{i,j}$ of the Steiner surface.
    For any $0\leq i<j\leq 3$, the point $z^{i,j}$ is given by $z_i=z_j$ and $z_k=0$ if $k\neq i,j$. We point out that, unlike in the toric case \cite[Theorem 2.36]{esterov2013discriminant}, the drop $k$ of $\chi_z = 6-k$ at a generic point $z$ of a codimension-one stratum $V(\mathfrak{p})$ is \emph{not} equal to the corresponding multiplicity $m_F$ in Corollary \ref{cor:factorization}.
\end{example}

\section{Open questions} \label{sec:questions}

Just like principal $A$-determinants, principal matroid determinants can be studied from many different perspectives. In this section, we collect some questions that are left unanswered. 

By \cite[Chapter 10, Theorem 1.2]{GKZbook}, the principal $A$-determinant factors as $E_A = \prod_{Q \preceq P_A} \Delta(X_{A_Q})^{m_Q}$, where $m_Q \geq 1$ is the multiplicity of the toric variety $X_A$ at each point of the torus orbit corresponding to the face $Q$. We conjecture the following analogous statement.

\begin{conjecture} \label{conj:multiplicities}
   Consider the factorization of $E_L$ from Corollary \ref{cor:factorization}. For each flat $F \in {\cal F}(M(L))$ such that $\nabla(L^{-1} \cap \Lambda_F)$ is a hypersurface, the exponent $m_F$ equals the local multiplicity of $L^{-2}$ along the stratum $L^{-2} \cap T_F$.
\end{conjecture}

The next conjecture characterizes dual defectivity in terms of the matroid $M(L)$. 

\begin{conjecture} \label{conj:dualdefectivity}
    The reciprocal linear space $L^{-1}$ is dual defective, i.e., $\dim \nabla (L^{-1}) < n-1$, if and only if the matroid $M(L)$ is not connected. 
\end{conjecture}

This conjecture was first proposed by Clara Briand. It has the following implication. The \emph{minimal building set} of the matroid $M(L)$, in the sense of \cite[page 470]{DeConciniProcesi1995WonderfulModels}, is the set of connected flats ${\cal B} = \{F \in {\cal F}(M(L)) \, : \, M(L_F) \text{ is connected} \}$. If Conjecture \ref{conj:dualdefectivity} is true, then the Newton polytope of the principal matroid determinant is a sum of simplices indexed by ${\cal B}$: 
\begin{equation} \label{eq:buildingset} 
{\rm Newt}(E_L) \, = \, \sum_{F \in {\cal B}} m_F \, \deg(\Delta(L_F^{-1})) \, \Delta_F,\end{equation}
see Theorem \ref{thm:newtEL}. 
For the boolean matroid, this is a \emph{nestohedron} as in \cite[Definition 6.3]{PostnikovReinerWilliams2008}.

The next conjecture pertains to the degree of $\nabla(L^{-1})$. It generalizes Theorem \ref{thm:uniform}.
\begin{conjecture} \label{conj:degree}
    If the matroid $M(L)$ is connected, then $\nabla(L^{-1})$ is a hypersurface of degree $2^d \, \beta(L)$, where $\beta (L)$ is the beta invariant of $M(L)$.
\end{conjecture}
If Conjecture \ref{conj:degree} holds, then we can set $\deg(\Delta(L_F^{-1})) = 2^d \, \beta(L)$ in \eqref{eq:buildingset}.

We have shown that $E_L$ is the Euler discriminant for a family of hypersurfaces $V_z$ in $\mathbb{P}^d\setminus {\cal A} \simeq L \cap T$ (Theorem \ref{thm:ED}). For all choices of $z$ at which $E_L(z) \neq 0$, the Euler characteristic of $V_z$ is constant. In the toric setting, the generic Euler characteristic is $\pm$ the volume of the convex hull of $A$. We do not know the generic value of $\chi(V_z)$ in the matroid setting. 

\begin{question} \label{q:genericeuler}
    What is the topological Euler characteristic of the hypersurface 
    \[ V_z \, = \, \Big \{ t \in \mathbb{P}^d \setminus {\cal A} \, : \, \frac{z_0}{\ell_0(t)} + \cdots + \frac{z_n}{\ell_n(t)} \, = \, 0 \Big \} \]
    for generic values of $z$? In particular, what is the formula in terms of matroid invariants? 
\end{question}

The (signed) Euler characteristic from Question \ref{q:genericeuler} is the \emph{maximum likelihood degree} in algebraic statistics \cite[Theorem 1.7]{huh2014likelihood}. For toric statistical models, Esterov's theorem \cite[Theorem 1.8]{esterov2013discriminant} mentioned above was rediscovered in \cite{amendola2019maximum}. In our setting, one could ask whether the complement of a hyperplane in $L^{-1} \cap T$ is a meaningful statistical model. 

The remaining questions concern the matroid hypergeometric system from Section \ref{sec:5}.

\begin{conjecture}\label{conj:Integral representation}
    For generic choices of $u\in\C^{n+1}$,  the morphism $\psi:M_L(u)\to M^{\rm hyp}_L(u)$ defined in Proposition \ref{prop:psi} is an isomorphism of $D_Z$-modules.
\end{conjecture}

Conjecture \ref{conj:Integral representation} means that any solution to $M_L(u)$ is given by an Euler integral \eqref{eq:Euler integral} for an appropriate choice of the twisted cycle $\Gamma$.
In the toric setting this is \cite[Theorem 2.10]{GKZ} where the genericity condition on the parameters $u$ is called \emph{non-resonance}. This raises the question what \emph{resonance} means in the setting of this paper. 

By our results in Section \ref{sec:6}, Conjecture \ref{conj:Integral representation}, if true, has the following consequences.

\begin{conjecture} \label{conj:bigconjecture}
    For generic choices of $u \in \mathbb{C}^{n+1}$, we have 
    \begin{enumerate}
        \item The holonomic rank of the matroid hypergeometric system $M_L(u)$ is the absolute value of the topological Euler characteristic of $(\mathbb{P}^d \setminus {\cal A}) \setminus V_z\simeq (L^{-1}\cap T)\setminus H_z$. That is, it equals $\beta(L) + (-1)^{d-1} \chi^*$, where $\chi^*$ is the generic Euler characteristic from Question \ref{q:genericeuler}.
        \item The characteristic variety of the matroid hypergeometric system $M_L(u)$ is $W$ from  \eqref{eq:W}.
        \item  The singular locus of $M_L(u)$ is the vanishing locus of the pricipal matroid determinant $V(E_L)$ in $\C^{n+1}$.
        \item $M_L(u)$ is a regular holonomic $D_Z$-module.
    \end{enumerate}
\end{conjecture}

In \cite[Chapter 3]{saito2013grobner}, the authors present a combinatorial strategy for constructing series solutions to the GKZ system. It would be interesting to develop a similar strategy for the matroid hypergeometric system $M_L(u)$.

\begin{question}
    How to construct series solutions to the differential equations \eqref{eq:op1}-\eqref{eq:op3}?
\end{question}

The multiplicities of the characteristic cycle of the GKZ system are determined in terms of a combinatorial invariant called the \emph{subdiagram volume} \cite[Theorem~5]{gel1989hypergeometric}.
We determined the characteristic variety of $M^{\rm hyp}_L(u)$ in Theorem \ref{thm:characteristic variety}, but we do not know the multiplicities.

\begin{question}
    What are the characteristic cycles of $M_L(u)$ and $M^{\rm hyp}_L(u)$?
    More specifically, what are the multiplicities of the conormal varieties ${\rm Con}(L^{-1}\cap\Lambda_F)$ in the characteristic~cycles?
\end{question}

\small

\paragraph{Acknowledgements.}
We thank Clara Briand, Leonie Kayser and Julian Weigert for useful discussions. S.-J.~M.-H.~thanks the Max Planck Institute for Mathematics in the Sciences in Leipzig for its hospitality during two research visits which made this project possible.

\bibliographystyle{abbrv}
\bibliography{references}

\bigskip
\noindent

\begin{tabular}{l}
    \textbf{Saiei-Jaeyeong Matsubara-Heo} \\
    Graduate School of Information Sciences, Tohoku University\\
    6-3-09 Aramaki-Aza-Aoba, Aoba-ku,     Sendai 980-8579, Japan \\
    \texttt{saiei@tohoku.ac.jp}
\end{tabular}

\bigskip
\noindent

\begin{tabular}{l}
    \textbf{Simon Telen} \\
    Max Planck Institute for Mathematics in the Sciences  \\
    Inselstrasse 22, 04103 Leipzig, Germany \\
    \texttt{simon.telen@mis.mpg.de}
\end{tabular}

\end{document}